\documentclass{amsart}
\usepackage{amsfonts}
\usepackage{appendix}

\newtheorem{theorem}{Theorem}
\theoremstyle{plain}

\newtheorem{corollary}{Corollary}

\newtheorem{definition}{Definition}

\newtheorem{lemma}{Lemma}

\newtheorem{proposition}{Proposition}
\newtheorem{remark}{Remark}

\numberwithin{equation}{section}
\input{tcilatex}

\begin{document}
\title[quadratic harnesses]{A few remarks on quadratic harnesses}
\author{Pawe\l\ J. Szab\l owski}
\address{Department of Mathematics and Information Sciences,\\
Warsaw University of Technology\\
ul Koszykowa 75, 00-662 Warsaw, Poland }
\email{pawel.szablowski@gmail.com}
\date{June 2012}
\subjclass[2000]{Primary 39A10, 39A13; Secondary 33D45, 60G44}
\keywords{system of difference equations, orthogonal polynomials, $q$%
-series, 3-term recurrence}
\thanks{The author would like to thank unknown referee for pointing out
numerous mistakes, inaccuracies and misprints of the first version of the
paper. }

\begin{abstract}
We analyze and partially solve system of recurrences that can be derived
from the properties of martingale orthogonal polynomials that characterize
quadratic harnesses (QH). We also specify conditions for the existence of
moments of one dimensional distribution for large classes of quadratic
harnesses that are also Markov processes complementing earlier results.
\end{abstract}

\maketitle

\section{Introduction}

In the series of papers \cite{brwe05}, \cite{BryMaWe}, \cite{BryMaWe07}, 
\cite{BryWe}, \cite{BryWe10}, \cite{BryMaWe11}, \cite{BryWe12} Bryc and Weso%
\l owski supported from time to time by Matysiak have defined and analyzed a
wide class of stochastic processes that they have called Quadratic Harnesses
(briefly QH). While performing this analysis they have obtained a system of $%
5$ recurrence equations depending on $5$ parameters.

In this paper we derive some general properties of this system of equations
as well as partially solve it. More precisely we are able to find solutions
of these recurrences mostly when 2 out of 5 parameters are set to zero.

In order to make paper complete we present the origins and interpretation of
these recurrences. We have placed them in the appendix. We also pose some
open problems that may encourage more gifted researches to continue our
investigations.

The paper is organized as follows.

In the next Section \ref{anal} we study these equations partially solving
them or at least providing conditions for the existence of one dimensional
measure that is identifiable by moments. On the way we use results of the
special auxiliary Section \ref{aux} that we have placed at the end of the
paper. Section \ref{anal} contains our main results concerning integrability
and existence of sequence of OMP for a given QH characterized by $5$
parameters.

Next short Section \ref{open} contains some remarks concerning the results
and presents some open problem that appeared while writing the paper. Final,
fifth Section \ref{dowody} contains lengthy proofs.

At the end we present an appendix \ref{QH}. We give there a definition and
basic properties of quadratic harnesses. That is we recollect results of
works of Bryc and others. We sketch an alternative derivation of the system
of $5$ iterative equations that must be satisfied if sequence of OMP is to
exist for a given QH. This derivation is believed to be simpler than the
original one given in \cite{BryMaWe07}.

Throughout the paper we use traditional notation used in $q-$series theory. $%
q$ is a parameter and we assume $q\in (-1,1].$ In particular we denote:

\begin{equation*}
\left[ 0\right] _{q}\allowbreak =\allowbreak 0;\left[ n\right]
_{q}\allowbreak =\allowbreak 1+q+\ldots +q^{n-1}\allowbreak ,\left[ n\right]
_{q}!\allowbreak =\allowbreak \prod_{j=1}^{n}\left[ j\right] _{q},
\end{equation*}%
with $\left[ 0\right] _{q}!\allowbreak =1,$ and%
\begin{equation*}
\QATOPD[ ] {n}{k}_{q}\allowbreak =\allowbreak \left\{ 
\begin{array}{ccc}
\frac{\left[ n\right] _{q}!}{\left[ n-k\right] _{q}!\left[ k\right] _{q}!} & 
, & n\geq k\geq 0, \\ 
0 & , & otherwise.%
\end{array}%
\right.
\end{equation*}%
The system of recurrences that we are going to study is the following:%
\begin{gather}
\tau \alpha _{n}\alpha _{n+1}+q\alpha _{n}\beta _{n+1}+\sigma \beta
_{n}\beta _{n+1}=\alpha _{n+1}\beta _{n},  \label{_1} \\
\tau \varepsilon _{n-1}\varepsilon _{n}+q\varepsilon _{n}\varphi
_{n-1}+\sigma \varphi _{n}\varphi _{n-1}=\varepsilon _{n-1}\varphi _{n},
\label{_2} \\
\theta \alpha _{n}+\eta \beta _{n}+\tau \alpha _{n}(\gamma _{n}+\gamma
_{n+1})+\sigma \beta _{n}(\delta _{n}+\delta _{n+1})+q(\alpha _{n}\delta
_{n+1}+\beta _{n}\gamma _{n})=\beta _{n}\gamma _{n+1}+\alpha _{n}\delta _{n},
\label{_3} \\
\theta \varepsilon _{n}+\eta \varphi _{n}+\tau \varepsilon _{n}(\gamma
_{n}+\gamma _{n-1})+\sigma \varphi _{n}(\delta _{n-1}+\delta _{n})+q(\varphi
_{n}\gamma _{n}+\delta _{n-1}\varepsilon _{n})=\varepsilon _{n}\delta
_{n}+\varphi _{n}\gamma _{n-1},  \label{_4} \\
1+\theta \gamma _{n}+\eta \delta _{n}+\tau \gamma _{n}^{2}+\sigma \delta
_{n}^{2}+\tau (\alpha _{n-1}\varepsilon _{n}+\alpha _{n}\varepsilon
_{n+1})+\sigma (\varphi _{n}\beta _{n-1}+\beta _{n}\varphi _{n+1})
\label{_5} \\
+q(\gamma _{n}\delta _{n}+\beta _{n-1}\varepsilon _{n}+\alpha _{n}\varphi
_{n+1})=\gamma _{n}\delta _{n}+\beta _{n}\varepsilon _{n+1}+\varphi
_{n}\alpha _{n-1}.  \label{_6}
\end{gather}%
for $n\geq 1$ with initial conditions $\alpha _{0}\allowbreak =\allowbreak
\gamma _{0}\allowbreak =\allowbreak \delta _{0}\allowbreak =\allowbreak
\varphi _{1}\allowbreak =\allowbreak 0$ and $\beta _{0}\allowbreak
=\allowbreak \varepsilon _{1}\allowbreak =\allowbreak 1.$

The general task is to examine existence of solutions of this system of
recurrences for all values of parameters, since these recurrences may be
interesting on its own. However from the point of view of the theory of
quadratic harness the most important are those ranges of parameters that
lead to sets of orthogonal polynomials more specifically mostly to the
polynomials from the so called Askey--Wilson scheme of polynomials.

The meaning of real sequences $\left\{ \alpha _{n}\right\} ,$ $\left\{ \beta
_{n}\right\} ,$ $\left\{ \gamma _{n}\right\} ,$ $\left\{ \delta _{n}\right\}
,$ $\left\{ \varepsilon _{n}\right\} ,$ $\left\{ \varphi _{n}\right\} $ is
the following. Using these sequences we define $3$ sequences of linear
functions $a_{n}(t)\allowbreak =\allowbreak \alpha _{n}t+\beta _{n},$ $%
b_{n}(t)\allowbreak =\allowbreak \gamma _{n}t+\delta _{n},$ $%
c_{n}(t)\allowbreak =\allowbreak \varepsilon _{n}t+\varphi _{n},$ for $t\geq
0$ so that these function sequences define a system of orthogonal
polynomials defined by the following 3-term recurrence:%
\begin{equation}
xp_{n}(x;t)\allowbreak =\allowbreak
a_{n}(t)p_{n+1}(x;t)+b_{n}(t)p_{n}(x;t)+c_{n}(t)p_{n-1}(x;t),  \label{3tr}
\end{equation}%
for all $t\geq 0,$ $n\geq 0.$ It is convenient to interpret parameter $t$ as
time. Functions $a_{n},$ $b_{n}$ and $c_{n}$ are real and must depend on $t$
and on five parameters $\sigma ,\tau ,\eta ,\theta $ and $q$ only.

Now let us recall from the general theory of orthogonal polynomials that the
3-term recurrence (\ref{3tr}) defines the system of polynomials orthogonal
with respect to a nonnegative measure with the support consisting of
infinitely many points iff for all $n\geq 1$ and $t\geq 0:$ 
\begin{equation}
a_{n-1}(t)c_{n}(t)>0.  \label{exi}
\end{equation}%
For more details on the theory of orthogonal polynomials we refer the reader
for example to \cite{Akhizer65} or to \cite{Sim98}. Hence from this point of
view the most interesting range of parameters is the one that guaranties
fulfilment of this requirement. Further task is to solve recurrences and
then having solutions find polynomials $\left\{ p_{n}\right\} $ solving
above mentioned 3-term recurrence and finally find positive measures that
make these polynomials orthogonal.

In the next section we will define the some ranges of parameters $q,$ $\tau
, $ $\sigma ,$ $\theta ,$ $\eta $ that assure that condition (\ref{exi}) is
satisfied.

\section{Analysis and integrability\label{anal}}

We have the following two observations concerning the first and the second
pair of the system of equations (\ref{_1})-(\ref{_6}). The first one
concerns equations (\ref{_1}) and (\ref{_2}). Both these observations are
the base of the following simple remark. Notice that if the OMP sequence is
to exist coefficients $a_{n}(t)\neq 0$ for all $n$ and $t$. That is
sequences $\{\alpha _{n}\}$ and $\{\beta _{n}\}$ cannot vanish
simultaneously. Similarly $c_{n}(t)$ cannot vanish for all $n$ and $t.$

In the sequel the crucial r\^{o}le will be played by the following sequence $%
\left\{ \lambda _{n}\right\} _{n\geq 0}$ defined by the following recursion: 
\begin{equation}
\lambda _{n+1}=\frac{1+q\lambda _{n}}{1-\sigma \tau \lambda _{n}}
\label{_lambda}
\end{equation}%
with $\lambda _{0}\allowbreak =\allowbreak 0.$

Below we present two general properties of our system of recurrences that
are true provided parameters are chosen in such a way that sequence of reals 
$\left\{ \lambda _{n}\right\} $ defined by (\ref{_lambda}) and sequence of
matrices defined by (\ref{A}) do exist. These properties will lead to
certain equivalent reformulation of recurrences simplifying imposition of
the condition (\ref{exi}). Then we will formulate appropriate ranges for
parameters $q,$ $\tau ,$ $\sigma ,$ $\theta ,$ $\eta .$

Hence for the moment let us assume that parameters $\sigma \tau $ and $q$
are chosen in such away that $\forall n\geq 1$ $\lambda _{n}\neq 0.$ Lemma %
\ref{existence}, below provides necessary conditions for this to happen.

\begin{theorem}
\label{firsttwo}$\forall n\geq 0:$ $\alpha _{n}\allowbreak =\allowbreak
\sigma \lambda _{n}\beta _{n}$ and $\varphi _{n}\allowbreak =\allowbreak
\tau \lambda _{n-1}\varepsilon _{n}.$
\end{theorem}

\begin{proof}
Firstly assume that $\sigma \allowbreak =\allowbreak 0$ and let us consider
equation (\ref{_1}). Put $n\allowbreak =\allowbreak 0.$ Then we see that $%
\alpha _{1}\allowbreak =\allowbreak 0.$ But this means that $\beta _{1}\neq
0.$ Now let us put $n\allowbreak =\allowbreak 1$ in (\ref{_1}). We have $%
0\allowbreak =\allowbreak \alpha _{2}\beta _{1}$. Hence we deduce that $%
\alpha _{2}\allowbreak =\allowbreak 0$ and $\beta _{2}\neq 0.$ And so on by
induction we deduce that in this case $\alpha _{n}\allowbreak =\allowbreak 0$
and $\beta _{n}\neq 0.$ In the similar way we show that if $\tau \allowbreak
=\allowbreak 0$ then $\varphi _{n}\allowbreak =\allowbreak 0$ and $%
\varepsilon _{n}\neq 0.$ Now assume that $\sigma \neq 0$ and consider
equation (\ref{_1}). Notice that for $n\allowbreak =\allowbreak 0$ this
equation yields $\alpha _{1}\allowbreak =\allowbreak \sigma \beta _{1}.$
Since $\alpha _{1}^{2}+\beta _{1}^{2}\neq 0$ we see that $\beta _{1}\neq 0.$
We will now deduce by induction. Assume that $\alpha _{n}\allowbreak
=\allowbreak \sigma \lambda _{n}\beta _{n}$ and $\beta _{n}\neq 0.$
Considering equation (\ref{_1}) we have 
\begin{equation*}
\alpha _{n+1}(1-\sigma \tau \lambda _{n})\beta _{n}\allowbreak =\allowbreak
(1+q\lambda _{n})\sigma \beta _{n}\beta _{n+1}.
\end{equation*}%
Since $\beta _{n}\allowbreak \neq 0$ we deduce that $\alpha
_{n+1}\allowbreak =\allowbreak \sigma \lambda _{n+1}\beta _{n+1}$ and $\beta
_{n+1}\neq 0$ if $\lambda _{n+1}\neq 0.$ We reason similarly in the case of
equation (\ref{_2}) and sequences $\varphi _{n}$ and $\varepsilon _{n}.$
\end{proof}

Further let us consider equations (\ref{_3}) and (\ref{_4}).

Again for the moment let us assume that the parameters $\sigma ,$ $\tau ,$ $%
q $ are chosen in such a way that matrices defined by (\ref{A}) are
non-singular. Lemma \ref{existence}, below will present necessary conditions
for this to happen. We have the following result.

\begin{theorem}
Let us define the following sequence of $2\times 2$ matrices for $n\geq 1$: 
\begin{eqnarray}
A_{n} &=&\left[ 
\begin{array}{cc}
1-\tau \sigma \lambda _{n} & -\sigma (1+q\lambda _{n}) \\ 
-\tau (1+q\lambda _{n}) & 1-\sigma \tau \lambda _{n}%
\end{array}%
\right] ,  \label{A} \\
B_{n}\allowbreak &=&\allowbreak \left[ 
\begin{array}{cc}
q+\sigma \tau \lambda _{n} & \sigma (1-\lambda _{n}) \\ 
\tau (1-\lambda _{n}) & q+\sigma \tau \lambda _{n}%
\end{array}%
\right] ,  \label{B} \\
C_{n} &=&\left[ 
\begin{array}{cc}
\sigma \lambda _{n} & 1 \\ 
1 & \tau \lambda _{n}%
\end{array}%
\right] .  \label{C}
\end{eqnarray}%
Then sequences $\left\{ \gamma _{n}\right\} $ and $\left\{ \delta
_{n}\right\} $ are given by the following formulae for $n\geq 0$:%
\begin{equation}
\left[ 
\begin{array}{c}
\gamma _{n+1} \\ 
\delta _{n+1}%
\end{array}%
\right] =\sum_{k=0}^{n}(\prod_{j=k+1}^{n}\Xi _{j})w_{k},  \label{rozw}
\end{equation}%
where we denoted $w_{k}\allowbreak =\allowbreak A_{k}^{-1}C_{k}\left[ 
\begin{array}{c}
\theta \\ 
\eta%
\end{array}%
\right] ,$ $\Xi _{k}\allowbreak =\allowbreak A_{k}^{-1}B_{k}$ for $k>0$. In (%
\ref{rozw}) we set $\prod_{k=n+1}^{n}\Xi _{k}\allowbreak =\allowbreak I$ and 
$\left[ 
\begin{array}{c}
\gamma _{1} \\ 
\delta _{1}%
\end{array}%
\right] \allowbreak \allowbreak =\allowbreak w_{0}.$

Let us denote $\chi _{n}\allowbreak =\allowbreak \beta _{n-1}\varepsilon
_{n}.$ Then sequence $\chi _{n}$ satisfies the following recursion : 
\begin{equation}
(q+\sigma \tau -\sigma \tau (1-\lambda _{n-1})^{2})\chi _{n}+1+\theta \gamma
_{n}+\tau \gamma _{n}^{2}+\eta \delta _{n}+\sigma \delta
_{n}^{2}-(1-q)\gamma _{n}\delta _{n}=(1-\sigma \tau (2\lambda _{n}+q\lambda
_{n}^{2}))\chi _{n+1},  \label{be}
\end{equation}%
with $\chi _{1}\allowbreak =\allowbreak 1.$
\end{theorem}

\begin{proof}
Dividing both sides of (\ref{_3}) by $\beta _{n}$ and keeping in mind that $%
\frac{\alpha _{n}}{\beta _{n}}\allowbreak =\allowbreak \sigma \lambda _{n}$
we get: 
\begin{equation*}
\gamma _{n+1}(1-\tau \sigma \lambda _{n})-\sigma (1+q\lambda _{n})\delta
_{n+1}=(\tau \sigma \lambda _{n}+q)\gamma _{n}+\sigma (1-\lambda _{n})\delta
_{n}+\theta \sigma \lambda _{n}+\eta .
\end{equation*}%
Dividing both sides of (\ref{_4}) by $\varepsilon _{n}$ and keeping in mind
that $\frac{\varphi _{n}}{\varepsilon _{n}}\allowbreak =\allowbreak \tau
\lambda _{n-1}$ we get after shifting indices by $1$: 
\begin{equation*}
\delta _{n+1}(1-\sigma \tau \lambda _{n})-\tau (1+q\lambda _{n})\gamma
_{n+1}=\tau (1+q\lambda _{n})\gamma _{n}+(\sigma \tau \lambda _{n}-1)\delta
_{n}+\theta +\tau \lambda _{n}\eta .
\end{equation*}%
As the result we obtain the following new vector form of equations (\ref{_3}%
) and (\ref{_4}): 
\begin{equation}
\left[ 
\begin{array}{c}
\gamma _{n+1} \\ 
\delta _{n+1}%
\end{array}%
\right] =A_{n}^{-1}B_{n}\left[ 
\begin{array}{c}
\gamma _{n} \\ 
\delta _{n}%
\end{array}%
\right] +A_{n}^{-1}C_{n}\left[ 
\begin{array}{c}
\theta \\ 
\eta%
\end{array}%
\right] ,  \label{_gd}
\end{equation}%
with $\left[ 
\begin{array}{c}
\gamma _{0} \\ 
\delta _{0}%
\end{array}%
\right] \allowbreak =\left[ 
\begin{array}{c}
0 \\ 
0%
\end{array}%
\right] \allowbreak $ for $n\geq 0$. Proof of (\ref{rozw}) is by induction.
We have on the left hand side: $\left[ 
\begin{array}{c}
\gamma _{n+1} \\ 
\delta _{n+1}%
\end{array}%
\right] \allowbreak =\allowbreak \sum_{k=0}^{n}(\prod_{j=k+1}^{n}\Xi
_{j})w_{k},$ while right hand side is $\Xi
_{n}\sum_{k=0}^{n-1}(\prod_{j=k+1}^{n-1}\Xi _{j}\allowbreak
)w_{k}\allowbreak +\allowbreak \allowbreak w_{n}\allowbreak =\allowbreak
\sum_{k=0}^{n-1}(\prod_{j=k+1}^{n}\Xi _{j})w_{k}+w_{n}\allowbreak
=\allowbreak \sum_{k=0}^{n}(\prod_{j=k+1}^{n}\Xi _{j})w_{k}.$ Finally let us
consider equation given by (\ref{_5}) and (\ref{_6}). Taking $\alpha _{n}$ $%
\allowbreak =\allowbreak \sigma \lambda _{n}\beta _{n}$ and $\varphi
_{n+1}\allowbreak =\allowbreak \tau \lambda _{n}\varepsilon _{n+1}$ and
denoting $\chi _{n}\allowbreak =\allowbreak \beta _{n-1}\varepsilon _{n}$ we
get (\ref{be}) with $\chi _{1}\allowbreak =\allowbreak 1.$
\end{proof}

As a corollary we have the following observation.

\begin{corollary}
\label{q-form}%
\begin{equation*}
1+\theta \gamma _{n}+\tau \gamma _{n}^{2}+\eta \delta _{n}+\sigma \delta
_{n}^{2}-(1-q)\gamma _{n}\delta _{n}\allowbreak =\allowbreak \mathbf{\mu }%
^{T}((\mathbf{D}_{n}+\mathbf{D}_{n}^{T})/2+\mathbf{D}_{n}^{T}\mathbf{\Delta D%
}_{n})\mathbf{\mu ,}
\end{equation*}
where we denoted $\mathbf{\mu \allowbreak =\allowbreak (}\theta ,\eta )^{T},$
$\mathbf{D}_{n}\allowbreak =\allowbreak \sum_{k=0}^{n}(\prod_{j=k+1}^{n}\Xi
_{j})A_{k}^{-1}C_{k}$ and 
\begin{equation*}
\Delta \allowbreak =\allowbreak \left[ 
\begin{array}{cc}
\tau & -(1-q)/2 \\ 
-(1-q)/2 & \sigma%
\end{array}%
\right] .
\end{equation*}
\end{corollary}

\begin{proof}
Notice that from (\ref{rozw}) it follows that vector $(\gamma _{n},\delta
_{n})^{T}$ is a linear form of parameters $\theta $ and $\eta $ more
precisely that $(\gamma _{n},\delta _{n})^{T}\allowbreak =\allowbreak 
\mathbf{D}_{n}\mathbf{\mu .}$ The final form follows the fact that $\mathbf{%
\mu }^{T}\mathbf{A\mu \allowbreak =\allowbreak }$ $\mathbf{\mu }^{T}(\mathbf{%
A+A}^{T})\mathbf{\mu /}2$ for any matrix $\mathbf{A.}$
\end{proof}

Let us now redefine orthogonal polynomials $p_{n}$ originally defined by (%
\ref{3tr}). Namely let us consider new polynomials $M_{n}\left( y|t,\sigma
,\tau ,\theta ,\eta ,q\right) $ briefly denoted $M_{n}\left( y\right) $
related to polynomials $p_{n}$ in the following way: 
\begin{equation}
M_{n}\left( y\right) =\frac{\prod_{j=0}^{n-1}a_{j}(t)}{t^{n/2}}p_{n}\left( y%
\sqrt{t};t\right) .  \label{_Mn}
\end{equation}%
Polynomials $M_{n}\left( y\right) $ although have less straightforward
probabilistic interpretation are easier to analyze. We have the following
simple observation.

\begin{proposition}
\label{new_pol}Polynomials $\left\{ M_{n}\right\} $ satisfy the following
3-term recurrence:%
\begin{equation}
yM_{n}\left( y\right) =M_{n+1}\left( y\right) +\left( \gamma _{n}\sqrt{t}%
+\delta _{n}/\sqrt{t}\right) M_{n}\left( y\right) +\chi _{n}(1+\sigma
\lambda _{n-1}t+\tau \lambda _{n-1}/t+\sigma \tau \lambda
_{n-1}^{2})M_{n-1}\left( y\right) ,  \label{n_3tr}
\end{equation}%
with $M_{-1}\left( y\right) \allowbreak =\allowbreak 0,$ $M_{0}\left(
y\right) \allowbreak =\allowbreak 1.$
\end{proposition}

\begin{proof}
Multiplying both sides of (\ref{3tr}) by $\prod_{j=0}^{n-1}a_{j}\left(
t\right) /t^{(n+1)/2}$ we get 
\begin{gather*}
\prod_{j=0}^{n-1}a_{j}\left( t\right) \frac{x}{\sqrt{t}}\frac{p_{n}\left(
x;t\right) }{t^{n/2}}=\prod_{j=0}^{n}a_{j}\left( t\right) \frac{%
p_{n+1}\left( x;t\right) }{t^{(n+1)/2}}+\prod_{j=0}^{n-1}a_{j}\left(
t\right) \frac{1}{t^{(n+1)/2}}b_{n}(t)p_{n}(x,t) \\
+a_{n-1}\left( t\right) \frac{c_{n}\left( t\right) }{t}%
\prod_{j=0}^{n-2}a_{j}\left( t\right) \frac{p_{n-1}\left( x;t\right) }{%
t^{(n-1)/2}}.
\end{gather*}%
Now it remains to change variable $x\longrightarrow y\sqrt{t},$ use Theorem %
\ref{firsttwo} and multiply $(\sigma \lambda _{n-1}t+1)(1+\tau \lambda
_{n-1}/t).$
\end{proof}

Now let us recall Favard's Theorem that assures that the infinitely
supported measure that makes polynomials $M_{n}$ orthogonal is positive iff $%
\forall t>0,~n\geq 0:$ 
\begin{equation*}
\chi _{n}(1+\sigma \lambda _{n-1}t+\tau \lambda _{n-1}/t+\sigma \tau \lambda
_{n-1}^{2})>0.
\end{equation*}%
First observation is that both $\forall n\geq 0:$ $\sigma \chi _{n}\lambda
_{n-1}\geq 0$ and $\tau \chi _{n}\lambda _{n-1}\geq 0$ from which it
immediately follows that $\sigma \tau \geq 0.$ However for the function $%
(1+\sigma \lambda _{n-1}t+\tau \lambda _{n-1}/t+\sigma \tau \lambda
_{n-1}^{2})$ not to change sign for $t>0$ one needs that both of its roots
be non-positive which implies that both $\sigma \lambda _{n-1}$ and $\tau
\lambda _{n-1}$ be non-negative. Notice further also that following
properties of the sequence $\left\{ \lambda _{n}\right\} $ we deduce that
the sequence $\left\{ (1+\sigma \lambda _{n-1}t+\tau \lambda _{n-1}/t+\sigma
\tau \lambda _{n-1}^{2})\right\} $ is bounded and positive for every $t>0.$
Hence if this measure has to be nonnegative the sequence $\left\{ \chi
_{n}\right\} $ has to be additionally positive. Taking into account the fact
that from assertion iv) of Proposition \ref{pomoc} it follows that if $q>1-2%
\sqrt{\sigma \tau }$ then the sequence $\lambda _{n}$ changes sign
infinitely often we deduce that then the sequence $\left\{ \chi
_{n}(1+\sigma \lambda _{n-1}t+\tau \lambda _{n-1}/t+\sigma \tau \lambda
_{n-1}\lambda _{n})\right\} $ cannot be non-negative for all $t>0.$ Hence we
will consider only the case $q\leq 1-2\sqrt{\sigma \tau }.$

\begin{remark}
\label{symm}Notice that redefining variables $x\longrightarrow y\sqrt{t}$
allows to see additional symmetries of the distribution that makes
polynomials $p_{n}$ orthogonal. Namely from formula (\ref{n_3tr}) it follows
that if $\sigma \allowbreak =\allowbreak \tau $ and $\eta \allowbreak
=\allowbreak \theta $ then the we have $\gamma _{n}\allowbreak =\allowbreak
\delta _{n}$ for all $n\geq 0$ and consequently that the change of time $%
t\longrightarrow 1/t$ does not change polynomials $M_{n}(y,t).$
\end{remark}

To clear the situation concerning existence of the sequence $\left\{ \lambda
_{n}\right\} _{n\geq 0},$ and invertibility of matrices $\left\{
A_{n}\right\} _{n\geq 0}$ we have the following lemma being a direct
consequence of Proposition \ref{pomoc} presented below in Section \ref{aux}.

\begin{lemma}
\label{existence}i) If $1\geq \sigma \tau \geq 0$ and $-1<q\leq 1-2\sqrt{%
\sigma \tau }$ then $\forall n\geq 0:$ $\lambda _{n}\allowbreak \geq
\allowbreak 0$ and 
\begin{equation*}
\lambda _{n}\longrightarrow y(q,\sigma \tau )\leq \frac{2}{1-q}\leq \frac{1}{%
\sqrt{\sigma t}},
\end{equation*}%
as $n\longrightarrow \infty ,$ where $y(q,\sigma \tau )$ is given by (\ref%
{gran}). Moreover if $-1<q<1-2\sqrt{\sigma \tau }$ then this convergence is
exponential i.e. $\left\vert \lambda _{n}-y(q,\sigma \tau )\right\vert
<C^{n}\left( q,\sigma \tau \right) $ where $C\in \lbrack 0,1).$

ii) Matrices $A_{n}$ defined by (\ref{A}) are non-singular.

iii) $\forall n\geq 1:$%
\begin{equation*}
(1-\sigma \tau (2\lambda _{n}+q\lambda _{n}^{2}))>0.
\end{equation*}%
If $q+\sigma \tau \geq 0$ then 
\begin{equation*}
(q+\sigma \tau -\sigma \tau (1-\lambda _{n-1})^{2})\geq 0,
\end{equation*}%
for all $n\geq 1$ and if $q+\sigma \tau <0$ then 
\begin{equation*}
(q+\sigma \tau -\sigma \tau (1-\lambda _{n-1})^{2})<0,
\end{equation*}%
for all $n\geq 1.$

iv) If $-1<q\leq 1-2\sqrt{\sigma \tau }$ then $\frac{(q+\sigma \tau -\sigma
\tau (1-\lambda _{n-1})^{2})}{(1-\sigma \tau (2\lambda _{n}+q\lambda
_{n}^{2}))}\longrightarrow D(q,\sigma \tau ),$ as $n\longrightarrow \infty ,$
where 
\begin{equation*}
D(q,\sigma \tau )\allowbreak =\allowbreak \frac{4(q+\sigma \tau )}{\left(
1+q+\sqrt{(1-q)^{2}-4\sigma \tau }\right) ^{2}}.
\end{equation*}%
$\allowbreak $ For $-1<q<1-2\sqrt{\sigma \tau }$ we have $\allowbreak
\left\vert D(q,\sigma \tau )\right\vert <1.$

v) $-1<q<1-2\sqrt{\sigma \tau }$ and $\theta \allowbreak =\allowbreak \eta
\allowbreak =\allowbreak 0$ then sequence $\chi _{n}\longrightarrow \frac{%
\left( 1-q+\sqrt{(1-q)^{2}-4\sigma \tau }\right) }{2\sqrt{(1-q)^{2}-4\sigma
\tau }}$ as $n\longrightarrow \infty .$ Moreover if $q+\sigma \tau \geq 0$
then $\forall n\geq 1:\chi _{n}\geq 0.$
\end{lemma}

\begin{proof}
Proved is moved to Section \ref{dowody}.
\end{proof}

Let us recall also Theorem 2.5.5 of \cite{IA} assuring that the measure that
makes polynomials $\left\{ M_{n}\right\} $ orthogonal is unique and
compactly supported if for every $t>0$ sequences $\left\{ \left( \gamma
_{n}t+\delta _{n}\right) \right\} $ and $\left\{ \chi _{n}(1+\sigma \lambda
_{n-1}t+\tau \lambda _{n-1}/t+\sigma \tau \lambda _{n-1}^{2})\right\} $ are
bounded. Hence the mentioned above theorem requires in fact only that
sequences $\left\{ \gamma _{n},\delta _{n},\chi _{n}\right\} $ are bounded
and sequence $\left\{ \chi _{n}\right\} $ be positive (to make the measure
positive).

Proposition below lists several easy cases when almost full solution is
possible. The other more complicated cases require separate analysis and
treatment.

\begin{proposition}
\label{special cases}i) If $\tau \allowbreak =\allowbreak \theta \allowbreak
=\allowbreak 0,$ then $\lambda _{n}\allowbreak =\allowbreak \lbrack n]_{q},$ 
$\gamma _{n}=[n]_{q}\eta ,\delta _{n}=0,~$~$,$ $\chi _{n}=[n]_{q}$ $for$ $%
n\geq 0$ or if $\sigma =\allowbreak $ $\eta \allowbreak =\allowbreak 0,$
then $\lambda _{n}\allowbreak =\allowbreak \lbrack n]_{q},\gamma _{n}=0,~$~$%
\delta _{n}=[n]_{q}\theta ,~\chi _{n}=[n]_{q}$ for $n\geq 0.$

ii) If $\tau \allowbreak =\allowbreak \eta \allowbreak =\allowbreak 0$ then $%
\lambda _{n}\allowbreak =\allowbreak \lbrack n]_{q},$ $\gamma
_{n}\allowbreak =\allowbreak \lbrack n]_{q}([n]_{q}+[n-1]_{q})\theta \sigma
, $ $\delta _{n}\allowbreak =\allowbreak \lbrack n]_{q}\theta ,$ $\chi
_{n}\allowbreak =\allowbreak \lbrack n]_{q}\allowbreak +\allowbreak \lbrack
n-1]_{q}^{2}[n]_{q}\theta ^{2}\sigma $ or if $\theta \allowbreak
=\allowbreak \sigma \allowbreak =0,$ then $\gamma _{n}\allowbreak
=\allowbreak \lbrack n]_{q}\eta ,$ $\delta _{n}\allowbreak =\allowbreak
\lbrack n]_{q}([n-1]_{q}+[n]_{q})\eta \tau ,$ $\chi _{n}\allowbreak
=\allowbreak \lbrack n]_{q}\allowbreak +\allowbreak \lbrack
n-1]_{q}^{2}[n]_{q}\eta ^{2}\tau ,$ for $n\geq 0.$

iii) If $\sigma \allowbreak =\allowbreak \tau \allowbreak =\allowbreak 0,$
then: $\lambda _{n}\allowbreak =\allowbreak \lbrack n]_{q}$, $\gamma
_{n}\allowbreak =\allowbreak \lbrack n]_{q}\eta $, $\delta
_{n}=[n]_{q}\theta $ and $\chi _{n}\allowbreak =\allowbreak \lbrack
n]_{q}+[n-1]_{q}[n-2]\theta \eta ,$ for $n\geq 0.$

iv) If $q\allowbreak =\allowbreak \sigma \allowbreak =\allowbreak 0$ then $%
\lambda _{n}\allowbreak =\allowbreak 1,$ $\gamma _{n}\allowbreak
=\allowbreak \eta ,$ $\delta _{n}\allowbreak =\allowbreak \theta +2\eta \tau
,$ $\chi _{n}\allowbreak =\allowbreak 1+\eta \theta +\eta ^{2}\tau $ for $%
n>1,$ or if $q\allowbreak =\allowbreak \tau \allowbreak =\allowbreak 0$ then 
$\lambda _{n}\allowbreak =\allowbreak 1,$ $\gamma _{n}\allowbreak
=\allowbreak \eta +2\sigma \theta ,$ $\delta _{n}\allowbreak =\allowbreak
\theta ,$ $\chi _{n}\allowbreak =\allowbreak 1+\eta \theta +\theta
^{2}\sigma ,$ for $n\geq 0.$

v) If $q\allowbreak =\allowbreak -\sigma \tau ,$ then $\lambda
_{n}\allowbreak =\allowbreak 1,$ $\gamma _{n}=\frac{\eta +2\theta \sigma
+\eta \sigma \tau }{(1-\sigma \tau )^{2}},\delta _{n}\allowbreak
=\allowbreak \frac{\theta +2\eta \tau +\theta \sigma \tau }{(1-\sigma \tau
)^{2}}$ and $\chi _{1}\allowbreak =\allowbreak \frac{1}{(1-\sigma \tau )^{2}}
$ and $\chi _{n}\allowbreak =\frac{1}{(1-\sigma \tau )^{2}}\allowbreak
+\allowbreak \allowbreak \frac{(\eta +\theta \sigma )(\theta +\eta \tau )}{%
(1-\sigma \tau )^{4}},$ for $n\geq 0.$

vi) If $q\allowbreak =\allowbreak 1-2\sqrt{\sigma \tau },$ then $\lambda
_{n}\allowbreak =\allowbreak \frac{n}{1+(n-1)\sqrt{\sigma \tau }}.$ Assuming
that $\theta \allowbreak =\allowbreak \eta \allowbreak =0\allowbreak $
sequence $\left\{ \chi _{n}\right\} $ is given by the formula:%
\begin{equation}
\chi _{n}=\frac{n(1+(n-2)\sqrt{\sigma \tau })^{2}(1+(n-3)\sqrt{\sigma \tau })%
}{(1-\sqrt{\sigma \tau })^{2}(1+2(n-1)\sqrt{\sigma \tau })(1+2(n-2)\sqrt{%
\sigma \tau })},  \label{q+2p2}
\end{equation}%
for $n\geq 1.$
\end{proposition}

\begin{proof}
Lengthy proof is shifted to Section \ref{dowody}.
\end{proof}

From the above considerations follows the following Lemma that contains
observations concerning polynomials $M_{n}$ and the distribution that makes
these polynomials orthogonal.

\begin{lemma}
\label{particular}i) If $-1<q\allowbreak <\allowbreak 1-2\sqrt{\sigma \tau }$
and either $q+\sigma \tau \geq 0$ and $\forall n\geq 1:1+\theta \gamma
_{n}+\tau \gamma _{n}^{2}+\eta \delta _{n}+\sigma \delta
_{n}^{2}-(1-q)\gamma _{n}\delta _{n}\geq 0$ or $q+\sigma \tau <0$ and $%
\forall n>0:$ $\chi _{n}\geq 0$ then distribution of $X_{t}$ is compactly
supported for $t>0.$

ii) If $\theta =\eta =0$ and $-\sigma \tau \leq q\leq 1-2\sqrt{\sigma \tau }%
, $ then for every $t>0$ the distribution of $X_{t}$ is symmetric and
positive. Moreover if additionally $q\allowbreak <\allowbreak 1-2\sqrt{%
\sigma \tau },$ then it is also compactly supported.

iii) If $\tau \allowbreak =\allowbreak \theta \allowbreak =\allowbreak 0,$
then the polynomials $M_{n}$ defined by (\ref{_Mn}) satisfy the following
3-term recurrence:%
\begin{equation*}
yM_{n}\left( y\right) =M_{n+1}\left( y\right) +\sqrt{t}\eta \lbrack
n]_{q}M_{n}\left( y\right) +[n]_{q}(1+\sigma t[n]_{q})M_{n-1}\left( y\right)
,
\end{equation*}%
with $M_{-1}\left( y\right) \allowbreak =\allowbreak 0,$ $M_{0}\left(
y\right) \allowbreak =\allowbreak 1$ and if $\sigma \allowbreak =\allowbreak
0$ and $\eta =0,$ then the polynomials $M_{n}$ defined by (\ref{_Mn})
satisfy the following 3-term recurrence:%
\begin{equation*}
yM_{n}\left( y\right) =M_{n+1}\left( y\right) +\theta \lbrack
n]_{q}M_{n}\left( y\right) /\sqrt{t}+[n]_{q}(1+\tau \lbrack
n]_{q}/t)M_{n-1}\left( y\right) ,
\end{equation*}%
with $M_{-1}\left( y\right) \allowbreak =\allowbreak 0,$ $M_{0}\left(
y\right) \allowbreak =\allowbreak 1.$

For $\left\vert q\right\vert <1$ the measures that make these polynomials
orthogonal is compactly supported.

iv) If $\eta \allowbreak =\allowbreak \tau \allowbreak =\allowbreak 0$, then
the polynomials $M_{n}$ defined by (\ref{_Mn}) satisfy the following 3-term
recurrence:%
\begin{eqnarray}
yM_{n}\left( y\right) &=&M_{n+1}\left( y\right) +[n]_{q}\theta (t(\left[ n%
\right] _{q}+\left[ n-1\right] _{q})\sigma +1/\sqrt{t})M_{n}\left( y\right)
\label{case3} \\
&&+\left[ n\right] _{q}\left( 1+\theta ^{2}\sigma \left[ n-1\right]
_{q}^{2}\right) (1+\left[ n\right] _{q}\sigma t)(M_{n-1}\left( y\right) , 
\notag
\end{eqnarray}%
with $M_{-1}\left( y\right) \allowbreak =\allowbreak 0,$ $M_{0}\left(
y\right) \allowbreak =\allowbreak 1$ while if $\sigma \allowbreak
=\allowbreak \theta \allowbreak =\allowbreak 0$ then the polynomials $M_{n}$
defined by (\ref{_Mn}) satisfy the following 3-term recurrence:%
\begin{eqnarray}
yM_{n}\left( y\right) &=&M_{n+1}\left( y\right) +[n]_{q}\eta (\sqrt{t}+(%
\left[ n\right] _{q}+\left[ n-1\right] _{q})\tau /\sqrt{t})M_{n}\left(
y\right)  \label{case4} \\
&&+\left[ n\right] _{q}\left( 1+\eta ^{2}\tau \left[ n-1\right]
_{q}^{2}\right) (1+\left[ n\right] _{q}\tau /t)(M_{n-1}\left( y\right) . 
\notag
\end{eqnarray}%
For $\left\vert q\right\vert <1$ the measures that make these polynomials
orthogonal are compactly supported. If $q\allowbreak =\allowbreak 1$ and $%
\sigma >0$ in case of (\ref{case3}) or

$\tau >0$ in case (\ref{case4}) measures that make polynomials orthogonal
might not be indentifiable by moments. It requires special investigation.

v) If $\sigma \allowbreak =\allowbreak \tau \allowbreak =\allowbreak 0,$
then the polynomials $M_{n}$ defined by (\ref{_Mn}) satisfy the following
3-term recurrence:%
\begin{equation*}
yM_{n}(y)\allowbreak =\allowbreak M_{n+1}(y)\allowbreak +\allowbreak \lbrack
n]_{q}(\theta /\sqrt{t}+\eta \sqrt{t})M_{n}(y)+([n]_{q}+[n-1]_{q}[n-2]_{q}%
\theta \eta )M_{n-1}(y),
\end{equation*}%
Again the measure that makes these polynomials orthogonal is positive if $%
1\geq -\theta \eta $ and is compactly supported for $\left\vert q\right\vert
<1.$

vi) If $q\allowbreak =\allowbreak \sigma \allowbreak =\allowbreak 0$ then
the polynomials $M_{n}$ defined by (\ref{_Mn}) satisfy the following 3-term
recurrence:%
\begin{equation*}
yM_{n}(y)\allowbreak =\allowbreak M_{n+1}(y)\allowbreak +\allowbreak (\eta 
\sqrt{t}+(\theta +2\eta \tau )/\sqrt{t})M_{n}(y)+\left( 1+\eta \theta +\eta
^{2}\tau \right) (1+\tau /t)M_{n-1}(y),
\end{equation*}%
with $M_{-1}\left( y\right) \allowbreak =\allowbreak 0,$ $M_{0}\left(
y\right) \allowbreak =\allowbreak 1$ while if $q\allowbreak =\allowbreak
\tau \allowbreak =\allowbreak 0$ then the polynomials $M_{n}$ defined by (%
\ref{_Mn}) satisfy the following 3-term recurrence:%
\begin{equation*}
yM_{n}(y)\allowbreak =\allowbreak M_{n+1}(y)\allowbreak +\allowbreak ((\eta
+2\sigma \theta )\sqrt{t}+\theta /\sqrt{t})M_{n}(y)+\left( 1+\eta \theta
+\theta ^{2}\sigma \right) (1+\sigma t)M_{n-1}(y),
\end{equation*}%
Again these measures is positive if either $\left( 1+\eta \theta +\eta
^{2}\tau \right) >0$ in the first case and $\left( 1+\eta \theta +\theta
^{2}\sigma \right) >0$ in the second.

vii) If $q\allowbreak =\allowbreak -\sigma \tau $, then the polynomials $%
M_{n}$ defined by (\ref{_Mn}) satisfy the following 3-term recurrence:%
\begin{eqnarray*}
yM_{n}(y)\allowbreak &=&\allowbreak M_{n+1}(y)\allowbreak +\allowbreak \frac{%
1}{(1-\sigma \tau )^{2}}(\left( \eta +2\theta \sigma +\eta \sigma \tau
\right) \sqrt{t}+(\theta +2\eta \tau +\theta \sigma \tau )/\sqrt{t})M_{n}(y)
\\
&&+\frac{1}{(1-\sigma \tau )^{2}}\left( 1\allowbreak +\allowbreak
\allowbreak \frac{(\eta +\theta \sigma )(\theta +\eta \tau )}{(1-\sigma \tau
)^{2}}\right) (1+\sigma \tau +\sigma t+\tau /t)M_{n-1}(y),
\end{eqnarray*}%
for $n>1$ with $M_{0}(y)\allowbreak =\allowbreak 1$ and 
\begin{equation*}
M_{1}(y)\allowbreak =\allowbreak y\allowbreak -\allowbreak \frac{1}{%
(1-\sigma \tau )^{2}}(\left( \eta +2\theta \sigma +\eta \sigma \tau \right) 
\sqrt{t}+(\theta +2\eta \tau +\theta \sigma \tau )/\sqrt{t}).
\end{equation*}%
Measure that makes polynomials $M_{n}$ orthogonal is positive and compactly
supported provided $\frac{(\eta +\theta \sigma )(\theta +\eta \tau )}{%
(1-\sigma \tau )^{2}}>-1.$
\end{lemma}

\begin{proof}
ii)-vii) follow directly from Proposition \ref{special cases} and conditions
for the measure orthogonalizing polynomials $M_{n}$ to be positive, while i)
follows Lemma \ref{existence},ii) iv) \& v). One has to keep in mind that in
order to assure that $\chi _{n}\geq 0$ and $\frac{(q+\sigma \tau -\sigma
\tau (1-\lambda _{n-1})^{2})}{(1-\sigma \tau (2\lambda _{n}+q\lambda
_{n}^{2}))}$ we do have to have $1+\theta \gamma _{n}+\tau \gamma
_{n}^{2}+\eta \delta _{n}+\sigma \delta _{n}^{2}-(1-q)\gamma _{n}\delta
_{n}\geq 0$
\end{proof}

\begin{remark}
i) If $\sigma \allowbreak =\allowbreak \tau \allowbreak =\allowbreak \theta
\allowbreak =\allowbreak \eta \allowbreak =\allowbreak 0$ then QH with these
parameters (defined e.g. by polynomials $p_{n}$ given by (\ref{3tr})) is $q-$%
Wiener process as described in \cite{Szab-OU-W} and \cite{BryMaWe07}.

ii) If $\sigma \allowbreak =\allowbreak \tau \allowbreak =\allowbreak \eta
\allowbreak =\allowbreak 0,$ and $q\allowbreak =\allowbreak \theta
\allowbreak =\allowbreak 1$ then QH with these parameters (defined e.g. by
polynomials $p_{n}$ given by (\ref{3tr})) is centered Poisson process as
described in \cite{BryMaWe07}. This statement is easier to verify if one
returns to variable $x\longrightarrow y\sqrt{t}$ which would result in the
following 3-term recurrence: 
\begin{equation*}
x\tilde{M}_{n}(x)\allowbreak =\allowbreak \tilde{M}_{n+1}(x)+n\tilde{M}%
_{n}(x)+n\tilde{M}_{n-1}(x),
\end{equation*}%
where we denoted $\tilde{M}_{n}(x)\allowbreak =\allowbreak M_{n}(y\sqrt{t}).$

iii) If we define generalized Chebyshev polynomials by $C_{n}(x)\allowbreak
=\allowbreak T_{n}((x-a)/b)$ then we see that polynomials $C_{n}$ satisfy
the following 3-term recurrence:%
\begin{equation*}
xC_{n}(x)\allowbreak =\allowbreak C_{n+1}(x)+aC_{n}(x)+b^{2}C_{n-1}(x).
\end{equation*}%
Comparing this formula with the ones given in assertions vi) and vii) of
Lemma \ref{particular} we deduce that respective processes are compactly
supported and we are able to describe completely their $1-$dimensional
distributions.
\end{remark}

\section{Remarks and open problems\label{open}}

\begin{enumerate}
\item We have solved system of recurrences (\ref{_1})-(\ref{_6}) for some
special values of parameters. Of course it would be interesting to find
general solution i.e. for all sensible values of parameters. By sensible
values we mean those assuring non-negativity of the sequence $\chi _{n}$
since then the measure that makes polynomials $M_{n}$ orthogonal is positive
i.e. probabilistic. Sensible might mean also that sequences $\gamma _{n},$ $%
\delta _{n}$ and $\chi _{n}$ are bounded since then the probability measure
would be compactly supported.

\item First of all notice that the set of allowed values of parameters $%
\sigma ,$ $\tau ,$ $\theta ,$ $\eta ,$ $q$ is such that $\sigma ,\tau \geq
0, $ $\theta ,\eta \in \mathbb{R}$ and $q\leq 1+2\sqrt{\sigma \tau }$ as
pointed out in \cite{BryMaWe07}. As it follows from the assertion iv) of
Proposition \ref{pomoc}, below if $q\in (1-2\sqrt{\sigma \tau },1+2\sqrt{%
\sigma \tau }]$ the sequence $\left\{ \lambda _{n}\right\} $ changes sign
infinitely often. Hence it is rather unlikely that a set of OMP defining
positive $1-$ dimensional measure exists. Is it really so? Do there exist QH
that are not Markovian and $q>1-2\sqrt{\sigma \tau }?$

\item It is not sure if parameters satisfying $\sigma ,\tau \geq 0,$ $\theta
,\eta \in \mathbb{R}$ and $q\leq 1+2\sqrt{\sigma \tau }$ or even $q\leq 1-2%
\sqrt{\sigma \tau }$ can guarantee that the sequence $\left\{ \chi
_{n}\right\} $ defined by (\ref{be}) is non-negative. It might not be true
in general since expression 
\begin{equation*}
1+\theta \gamma _{n}+\tau \gamma _{n}^{2}+\eta \delta _{n}+\sigma \delta
_{n}^{2}-(1-q)\gamma _{n}\delta _{n}
\end{equation*}%
can be written as a quadratic form in $(\theta ,\eta )$ as pointed out in
Corollary \ref{q-form}. However matrix of this form might not be positive
definite. This is so since matrix $\mathbf{\Delta }$ defined there is
non-positive definite hence, at least theoretically the values of this
expression can be negative and this may result in negative values of some
elements of the sequence $\left\{ \chi _{n}\right\} .$ Relatively simple
case occurs when $\theta \allowbreak =\allowbreak \eta =0.$ Then as we have
shown if $q+\sigma \tau \geq 0$ then sequence $\left\{ \chi _{n}\right\} $
is positive and bounded. If $q+\sigma \tau <0$ many numerical simulations
show that also in this case the sequence $\left\{ \chi _{n}\right\} $ is
positive. But proving it we leave to more gifted researchers.

\item The set of OMP of a given QH supplies knowledge about $1-$dimensional
distributions. But in fact knowing polynomials of OMP $\left\{ p_{j}\right\} 
$ we can also state something about transitional probability distribution.
Namely from the relationship (\ref{QMP}) we can also deduce that orthogonal
polynomials $\left\{ W_{n}\right\} $ of the transitional probability must be
of the form :%
\begin{equation*}
W_{n}(X_{t},t;X_{s},s)=\sum_{j=0}^{n}V_{n,j}\left( X_{s},s\right) \left(
p_{j}(X_{t},t)-p_{j}(X_{s},s)\right) ,
\end{equation*}%
for $s\leq t$ and some polynomials $V_{n,j}(X_{s},s)$ of order at most $n-j.$
This is so since we have $E_{x}W_{n}(X_{t},t;x,s)\allowbreak =\allowbreak 0$
for all $n\geq 1.$ Thus it remains to prove that these polynomials satisfy
some 3-term recurrence to be able to identify them as polynomials orthogonal
with respect to the transitional measure. It would be be interesting to find
this 3-term recurrence. The examples known so far suggest to seek
polynomials $V_{n,j}(X_{s},s)$ among such polynomials that :%
\begin{equation*}
\sum_{j=0}^{n}V_{n-j}(x;s)p_{j}(x,s)\allowbreak =\allowbreak 0,
\end{equation*}%
for a set of polynomials that are of order $n-j$ and indexed only one
integer index. Can it be true in the general case?

\item Finally notice that in Proposition \ref{special cases} we did not
consider cases $(\eta ,\theta )\allowbreak =\allowbreak (q,\eta )\allowbreak
=\allowbreak (q,\theta )\allowbreak =\allowbreak (0,0).$ Together with $7$
cases considered in assertions i)-iv) of this Proposition these $3$ cases
would make $10\allowbreak =\allowbreak \binom{5}{2}$ cases of $2$ out of $5$
parameters $(\sigma ,\tau ,\theta ,\eta ,q)$ that are set to zero. The cases 
$(\eta ,\theta )\allowbreak =\allowbreak (q,\eta )\allowbreak =\allowbreak
(q,\theta )\allowbreak =\allowbreak (0,0)$ were not considered because they
are too complicated and seem to require deeper, more detailed analysis. We
leave them to more talented researchers.
\end{enumerate}

\section{Auxiliary results\label{aux}}

In this section we will analyze properties of the sequence $\left\{ \lambda
_{n}\right\} $ that appear when examining properties of marginal
distributions of MQH. As the considerations presented above show it plays a
crucial r\^{o}le in the whole analysis.

One can easily observe that right hand side of (\ref{_lambda}) defines a
transformation of $\lambda _{n}$ which is known under the name of M\"{o}bius
transform. Much is known about its properties. However we need only those
properties of it that fit to our special setting. Thus it seems to be easier
and more logical to prove these needed properties once more than to dig in
the literature and reduce general cases to our setting. All the more the
proofs in these special cases are elementary although not quite simple.

\begin{proposition}
\label{pomoc}Let us denote $f(x|q,z)=\frac{1+qx}{1-zx}$ and assume $z\geq 0.$

i) If $q+z\geq 0$ and $x<1/z$ then $f(x|q,z)\geq x\geq 0.$ In particular if $%
q+z\allowbreak =\allowbreak 0$ then $f(x|q,z)\allowbreak =\allowbreak 1,$ if 
$q+z<0$ then for $0\leq x\leq \frac{-1}{q},$ $0\leq f(x|q,z)\leq x.$ In
particular if $q\allowbreak =\allowbreak -1$ we have $f(0|q,z)\allowbreak
=\allowbreak 1$ and $f(1|q,z)\allowbreak =\allowbreak 0.$

ii) If $z\in \lbrack 0,1)$ and $q\in \lbrack -1,1-2\sqrt{z}]$ then for $x\in
\lbrack 0,\frac{1}{\sqrt{z}})$ implies that $f(x|q,z)\in \lbrack 0,\frac{1}{%
\sqrt{z}}).$

iii) Let $f^{\left( n\right) }$ denote $n-$ fold composition of function $f.$
If $q\leq 1-2\sqrt{z}$ or $q\geq 1+2\sqrt{z}$ then for every $n$ there
exists a number $y_{n}$ such that $y_{n}\allowbreak =\allowbreak f^{\left(
n\right) }(y_{n}|q,z).$ Otherwise such number does not exist. Moreover if
they do exist all numbers $y_{n}$, are identical and equal to%
\begin{equation}
y(q,z)=\frac{2}{1-q+\sqrt{(1-q)^{2}-4z}}.  \label{gran}
\end{equation}

iv) If $z\in \lbrack 0,1)$ and $q\in \lbrack -1,1-2\sqrt{z})$ and $x\in
\lbrack 0,\frac{1}{\sqrt{z}})$ if $q+z\geq 0$ and $x\in \lbrack 0,\frac{1}{%
\left\vert q\right\vert })$ when $q+z<0$ we have: 
\begin{equation*}
\left\vert f(x|q,z)-y(q,z)\right\vert <C(q,z)|x-y(q,z)|,
\end{equation*}%
with $C\left( q,z\right) <1.$

v) If $q>1-2\sqrt{z}$ then sequence $\left\{ \lambda _{n}\right\} $ changes
sign infinitely many times.
\end{proposition}

\begin{proof}
First of all notice that $f(x|q,z)\allowbreak =\allowbreak 1\allowbreak
+\allowbreak \frac{(q+z)x}{1-zx}$ and $f^{\prime }(x|q,z)\allowbreak
=\allowbreak \frac{q+z}{(1-zx)^{2}}.$ i) If $q+z\geq 0$ then the derivative
of the function $f(x)$ is nonnegative. If $q\allowbreak +\allowbreak
z\allowbreak =\allowbreak 0$ we see that $f(x|q,z)\allowbreak -\allowbreak
1\allowbreak =\allowbreak 0.$ If $q\allowbreak +\allowbreak z<0$ analyzing
polynomial $(1+qx)(1-zx)$ we see that function $f(x|q,z)$ is decreasing and
nonnegative on $0\leq x\leq \frac{-1}{q}$ hence the first assertion is true.

ii) If $q+z<0$ then we have for$\frac{1}{\sqrt{z}}$ $>x\geq 0$ we have $%
f(x)\allowbreak =\allowbreak 1\allowbreak +\allowbreak \frac{(q+z)x}{(1-xz)}%
\,<1\leq \frac{1}{\sqrt{z}}.$ If $q+z\allowbreak =\allowbreak 0$ then $%
f(x)\allowbreak =\allowbreak 1<\frac{1}{\sqrt{z}}.$ If $q+z>0$ which is
equivalent in this case to $0<\left( q+z\right) \allowbreak \leq \allowbreak
(1-\sqrt{z})^{2}$ then $f(x)$ is increasing and we have $f(x)<f\left( \frac{1%
}{\sqrt{z}}\right) \allowbreak =\allowbreak 1\allowbreak +\allowbreak \frac{%
(q+z)}{\sqrt{z}\left( 1-\sqrt{z}\right) }\allowbreak \leq \allowbreak
1\allowbreak +\allowbreak \frac{1-\sqrt{z}}{\sqrt{z}}\allowbreak
=\allowbreak \frac{1}{\sqrt{z}}.\allowbreak $

iii) If $z=0$ then $f\left( x\right) \allowbreak =\allowbreak 1+qx,$ hence $%
f^{\left( n\right) }\left( x\right) \allowbreak =\allowbreak \lbrack
n-1]_{q}+q^{n}x,$ consequently $y_{n}\allowbreak =\allowbreak \frac{1}{1-q}.$
Assume that $z\neq 0.$ Let us notice that 
\begin{equation*}
f^{\left( n\right) }(x)\allowbreak =\allowbreak \frac{A_{n}+B_{n}x}{%
C_{n}-D_{n}x},
\end{equation*}%
for some depending on $q$ and $z$ functions $A_{n},$ $B_{n},$ $C_{n},$ $%
D_{n} $. Notice that the solution of the equation $f^{\left( n\right)
}(y_{n})\allowbreak =\allowbreak y_{n}$ satisfies the quadratic equation:%
\begin{equation*}
D_{n}y^{2}+(B_{n}-C_{n})y+A_{n}=0.
\end{equation*}%
Since $f\left( f^{\left( n\right) }\left( x\right) \right) \allowbreak
=\allowbreak f^{\left( n\right) }\left( f\left( x\right) \right) $ for every 
$x$ we deduce that:%
\begin{eqnarray*}
A_{n}+B_{n}\allowbreak &=&\allowbreak qA_{n}+C_{n}, \\
-zA_{n}+qB_{n} &=&qB_{n}-D_{n} \\
C_{n}-D_{n} &=&C_{n}-zA_{n}, \\
C_{n}z+qD_{n} &=&D_{n}+zB_{n},
\end{eqnarray*}%
and consequently that $B_{n}-C_{n}\allowbreak =\allowbreak (1-q)A_{n}$ and $%
D_{n}\allowbreak =\allowbreak zA_{n}.$ Since $A_{n}\neq 0$ (otherwise we
would have $f^{\left( n\right) }(x)\equiv x$ ) we deduce that for all $n$
number $y_{n}$ satisfies equation 
\begin{equation*}
zy^{2}+(1-q)y+1=0.
\end{equation*}%
Moreover real solution of this equation exists if $(1-q)^{2}\geq 4z$ or
equivalently if $q\leq 1-2\sqrt{z}$ or $q\geq 1+2\sqrt{z}.$ Now let us
consider the case $q\in (1-2\sqrt{z},1+2\sqrt{z}).$ Then as the above
analysis shows there is no solution of the equation $f^{(n)}(x)\allowbreak
=\allowbreak x$ for any $n>0.$

iv) First of all recall $-1+z\leq q+z\leq (1-\sqrt{z})^{2}$ and that $%
y(q,z)\allowbreak =\allowbreak \frac{1+qy(q,z)}{1-zy(q,z)}.$ Now assume that 
$q+z>0$ then 
\begin{equation*}
|f(x|q,z)-y(q,z)|\allowbreak =\allowbreak \frac{\left\vert
x-y(q,z)|\right\vert q+z|}{(1-zx)(1-zy(q,z)|}\leq C(q,z)|x-y(q,z)|,
\end{equation*}%
where $C(q,z)\allowbreak =\allowbreak \frac{\left\vert q+z)\right\vert }{(1-%
\sqrt{z})^{2}}<1$ since $y(q,z)\leq \frac{2}{1-q}\leq \frac{1}{\sqrt{z}}.$
If $q+z\allowbreak =\allowbreak 0$ then $f(x|q,z)\allowbreak =\allowbreak
y(-z,z)\allowbreak =\allowbreak 1.$ So it remains to consider the case $%
0>q+z>z-1.$ Recall that function $f$ is now decreasing and thus for $0\leq
x\leq \frac{-1}{q}$ we have:%
\begin{eqnarray*}
|f(x|q,z)-y(q,z)|\allowbreak &=&\allowbreak \frac{\left\vert
x-y(q,z)|\right\vert q+z|}{(1-zx)(1-zy(q,z))}\leq \frac{\left\vert
x-y(q,z)|\right\vert q+z|}{\left\vert \frac{q+z}{q}\right\vert (1-zy(q,z))}
\\
&\leq &|x-y(q,z)|C(q,z).
\end{eqnarray*}%
where $C(q,z)\allowbreak =\allowbreak \frac{\left\vert q\right\vert }{%
(1-zy(q,z))}\allowbreak =\allowbreak \left\vert q\right\vert \frac{1+qy(q,z)%
}{y(q,z)}\allowbreak =\allowbreak \left\vert q\right\vert \left\vert \frac{1%
}{y(q,z)}+q\right\vert \allowbreak =\allowbreak $\newline
$\left\vert q\right\vert \left\vert q+\frac{1}{2}-\frac{1}{2}q-\frac{1}{2}%
\sqrt{(1-q)^{2}-4z}\right\vert \allowbreak \leq \allowbreak \left\vert
q\right\vert \max (\left\vert \frac{1+q}{2}\right\vert ,\frac{(1-q)}{2}|<1.$

v) If $q>1-2\sqrt{z},$ then as it follows from assertion iii) of Proposition %
\ref{pomoc} there is no condensation point of the sequence $\left\{ \lambda
_{n}\right\} .$ Since in this case we have $q+z\allowbreak \geq \allowbreak
0 $ then the sequence $\left\{ \lambda _{n}\right\} $ is increasing and
consequently will will reach value more that $\frac{1}{z}.$ But then the
next iterate will be negative and again the sequence will be increasing and
so on.
\end{proof}

\section{Proofs\label{dowody}}

\begin{proof}[Proof of Lemma \protect\ref{existence}]
i) We base on Lemma \ref{pomoc} with $z\allowbreak =\allowbreak \sigma \tau $
and $q\in (-1,1-2\sqrt{\sigma \tau }).$ For $q\allowbreak =\allowbreak 1-2%
\sqrt{\sigma \tau }$ see the assertion vi) of Proposition \ref{special cases}%
.

ii) Notice that $\det A_{n}\allowbreak =\allowbreak (1-\sigma \tau \lambda
_{n})^{2}-\sigma \tau (1+q\lambda _{n})^{2}.$ If $\det A_{n}\allowbreak
=\allowbreak 0$ then we would have $\lambda _{n+1}^{2}\allowbreak
=\allowbreak \frac{1}{\sigma \tau }$ which cannot happen since assertion ii)
of Proposition \ref{special cases} states that if $\lambda _{n}\in \lbrack 0,%
\frac{1}{\sqrt{\sigma \tau }})$ then also $\lambda _{n+1}\in \lbrack 0,\frac{%
1}{\sqrt{\sigma \tau }})$ and we start from $\lambda _{0}\allowbreak
=\allowbreak 0.$

iii) We have $(1-\sigma \tau (2\lambda _{n}+q\lambda _{n}^{2}))\allowbreak
=\allowbreak (1-\sigma \tau \lambda _{n})^{2}-\sigma \tau (\sigma \tau
+q)\lambda _{n}^{2}.$ hence $(1-\sigma \tau (2\lambda _{n}+q\lambda
_{n}^{2}))\geq 0$ is equivalent to $1\allowbreak \geq \allowbreak (\sqrt{%
\sigma \tau }\allowbreak +\allowbreak \sqrt{q+\sigma \tau })\allowbreak 
\sqrt{\sigma \tau }\lambda _{n}.$ Now keeping in mind that $\ \sqrt{\sigma
\tau }\lambda _{n}\leq 1$ and $\sqrt{\sigma \tau }\allowbreak +\allowbreak 
\sqrt{q+\sigma \tau }\leq 1$ for $q\leq 1-2\sqrt{\sigma \tau }$ we see that
it is true. Now if $q+\sigma \tau \geq 0$ then $(q+\sigma \tau -\sigma \tau
(1-\lambda _{n-1})^{2})\geq 0$ is equivalent to $\sqrt{q+\sigma \tau }%
\allowbreak \geq \allowbreak \sqrt{\sigma \tau }(\lambda _{n}-1).$ But $%
\lambda _{n}-1\allowbreak =\allowbreak \frac{(q+\sigma \tau )\lambda _{n-1}}{%
1-\sigma \tau \lambda _{n-1}}.$ So the last inequality is true if $1-\sigma
\tau \lambda _{n-1}\geq \sqrt{\sigma \tau }\sqrt{q+\sigma \tau }\lambda
_{n-1}$ which is true by the previous argument.

iv) Denote $\sigma \tau $ as $z$ for simplicity. We will write $y$ instead
of $y(q,z)$ for simplicity. First of all notice that we have: $%
y-y^{2}z\allowbreak =\allowbreak 1+qy,$ so $zy^{2}\allowbreak =\allowbreak
y(1-q)-1.$ Next we have $q+2z\lambda _{n}-z\lambda _{n}^{2}\longrightarrow
q+2zy-y(q-1)+1\allowbreak =\allowbreak 1+q+2zy-y(1-q)$ and $(1-z(2\lambda
_{n}+q\lambda _{n}^{2}))\longrightarrow 1-2zy-q(y(1-q)-1)\allowbreak
=\allowbreak 1+q-2zy-qy(1-q)$ by assertion iv) of Lemma \ref{pomoc}. Now
keeping in mind that $\frac{1}{y(q,z)}\allowbreak =\allowbreak (1-q+\sqrt{%
(1-q)^{2}-4z})/2$ we get: 
\begin{gather*}
D\left( q,\sigma \tau \right) =\frac{1+q+2zy-y(1-q)}{1+q-2zy-qy(1-q)}%
\allowbreak \\
=\frac{(1+q)\sqrt{(1-q)^{2}-4z}+4z+2q-2+1-q^{2}}{(1+q)\sqrt{(1-q)^{2}-4z}%
+1-q^{2}-4z-2q+2q^{2}} \\
=\frac{1+q-\sqrt{(1-q)^{2}-4z}}{1+q+\sqrt{(1-q)^{2}-4z}}.
\end{gather*}%
Then we have on one side: $1+q-\sqrt{(1-q)^{2}-4z}<$ $1+q+\sqrt{(1-q)^{2}-4z}
$ for $4z<(1-q)^{2}$ and on the other side $1+q-\sqrt{(1-q)^{2}-4z}>-(1+q+%
\sqrt{(1-q)^{2}-4z})$ which is also true for $q>-1.$

v) Obviously sequence $\left\{ \chi _{n}\right\} $ is defined by the
recursion 
\begin{equation}
\chi _{n+1}\allowbreak =\allowbreak \kappa _{n}\chi _{n}+\frac{1}{(1-\sigma
\tau (2\lambda _{n}+q\lambda _{n}^{2}))},  \label{chin}
\end{equation}%
where we denoted $\kappa _{n}\allowbreak =\allowbreak \frac{(q+\sigma \tau
-\sigma \tau (1-\lambda _{n-1})^{2})}{(1-\sigma \tau (2\lambda _{n}+q\lambda
_{n}^{2}))}$. Next we notice that: (a) (\ref{chin}) is a inhomogeneous
linear recursive equation of the first degree, (b) as shown above sequences $%
\left\{ \kappa _{n}\right\} _{n\geq 1},$ $\left\{ \frac{1}{(1-\sigma \tau
(2\lambda _{n}+q\lambda _{n}^{2}))}\right\} _{n\geq 1}$ have limits and
moreover the limit of $\left\{ \kappa _{n}\right\} $ has absolute value less
than $1$. Now by standard argument used in the case of linear recursions we
deduce that $\chi _{n}$ has a limit $\xi $ and moreover this limit satisfies
equation 
\begin{equation*}
\xi =D(q,\sigma \tau )\xi +\frac{1-q+\sqrt{(1-q)^{2}-4z}}{\sqrt{(1-q)^{2}-4z}%
((1+q)+\sqrt{(1-q)^{2}-4z})}.
\end{equation*}
\newline
So 
\begin{gather*}
\xi \allowbreak =\allowbreak \frac{1-q+\sqrt{(1-q)^{2}-4z}}{\sqrt{%
(1-q)^{2}-4z}((1+q)+\sqrt{(1-q)^{2}-4z})}/(1-D(q,\sigma \tau ))\allowbreak \\
=\allowbreak \allowbreak \frac{1-q+\sqrt{(1-q)^{2}-4z}}{\sqrt{(1-q)^{2}-4z}%
((1+q)+\sqrt{(1-q)^{2}-4z})}\frac{1+q+\sqrt{(1-q)^{2}-4\sigma \tau }}{2\sqrt{%
(1-q)^{2}-4\sigma \tau }} \\
=\allowbreak \frac{\left( 1-q+\sqrt{(1-q)^{2}-4\sigma \tau }\right) }{2\sqrt{%
(1-q)^{2}-4\sigma \tau }}.
\end{gather*}
To see that $\chi _{n}\geq 0$ for all $n\geq 1$ we note that above we have
shown that if $q+\sigma \tau \geq 0$ the $\kappa _{n}\geq 0.$
\end{proof}

\begin{proof}[Proof of Proposition \protect\ref{special cases}]
First of all notice that if $\sigma \tau \allowbreak =\allowbreak 0$ then $%
\lambda _{n}\allowbreak =\allowbreak \lbrack n]_{q}$ since then equation (%
\ref{_lambda}) reduces to $\lambda _{n+1}\allowbreak =\allowbreak q\lambda
_{n}\allowbreak +\allowbreak 1,$ with $\lambda _{0}\allowbreak =\allowbreak
0.$

i) $\tau \allowbreak =\allowbreak \theta \allowbreak =\allowbreak 0$ Under
our assumptions we get $A_{n}\allowbreak =\allowbreak \left[ 
\begin{array}{cc}
1 & -\sigma \lbrack n+1]_{q} \\ 
0 & 1%
\end{array}%
\right] ,$ $B_{n}\allowbreak =\allowbreak \left[ 
\begin{array}{cc}
q & -\sigma q[n-1]_{q} \\ 
0 & q%
\end{array}%
\right] $ and $C_{n}\allowbreak =\allowbreak \left[ 
\begin{array}{cc}
\sigma \lbrack n]_{q} & 1 \\ 
1 & 0%
\end{array}%
\right] $ since $1+q[n]_{q}\allowbreak =\allowbreak \lbrack n+1]_{q}$ and $%
1-[n]_{q}\allowbreak =\allowbreak -q[n-1]_{q}.$ So $A_{n}^{-1}B_{n}%
\allowbreak =\allowbreak \allowbreak \left[ 
\begin{array}{cc}
q & \sigma (1+q)q^{n} \\ 
0 & q%
\end{array}%
\right] $ and $A_{n}^{-1}C_{n}\left[ 
\begin{array}{c}
0 \\ 
\eta%
\end{array}%
\right] \allowbreak =\allowbreak \left[ 
\begin{array}{c}
\eta \\ 
0%
\end{array}%
\right] .$ Hence vector $\left[ 
\begin{array}{c}
\gamma _{n} \\ 
\delta _{n}%
\end{array}%
\right] $ satisfies the following recursion: $\left[ 
\begin{array}{c}
\gamma _{n+1} \\ 
\delta _{n+1}%
\end{array}%
\right] \allowbreak =\allowbreak \left[ 
\begin{array}{cc}
q & \sigma (1+q)q^{n} \\ 
0 & q%
\end{array}%
\right] \left[ 
\begin{array}{c}
\gamma _{n} \\ 
\delta _{n}%
\end{array}%
\right] \allowbreak +\allowbreak \left[ 
\begin{array}{c}
\eta \\ 
0%
\end{array}%
\right] .$ So $\delta _{n}\allowbreak =\allowbreak 0,$ while $\gamma
_{n}\allowbreak =\allowbreak \lbrack n]_{q}\eta .$ Further we have\newline
$\left. 1+\theta \gamma _{n}+\tau \gamma _{n}^{2}+\eta \delta _{n}+\sigma
\delta _{n}^{2}-(1-q)\gamma _{n}\delta _{n}\right\vert _{\tau =0,\theta
=0,\gamma _{n}=\eta \lbrack n]_{q},\delta _{n}=0}\allowbreak =\allowbreak 1,$
so recursion (\ref{be}) reduces to%
\begin{equation*}
q\chi _{n}+1=\chi _{n+1},
\end{equation*}%
with $\chi _{1}\allowbreak =\allowbreak 1.$ Thus indeed $\chi
_{n}\allowbreak =\allowbreak \lbrack n]_{q}.$ If $\sigma =\eta \allowbreak
=\allowbreak 0$ we have symmetric situation.

ii) $\tau \allowbreak =\allowbreak \eta \allowbreak =\allowbreak 0.$ We get
then $A_{n}\allowbreak =\allowbreak \left[ 
\begin{array}{cc}
1 & -\sigma q[n+1]_{q} \\ 
0 & 1%
\end{array}%
\right] ,$ $B_{n}\allowbreak =\allowbreak \left[ 
\begin{array}{cc}
q & -\sigma q[n-1]_{q} \\ 
0 & q%
\end{array}%
\right] $ and $C_{n}\allowbreak =\allowbreak \left[ 
\begin{array}{cc}
\sigma \lbrack n]_{q} & 1 \\ 
1 & 0%
\end{array}%
\right] .$ Hence $A_{n}^{-1}B_{n}\allowbreak =\allowbreak \left[ 
\begin{array}{cc}
q & \sigma (1+q)q^{n} \\ 
0 & q%
\end{array}%
\right] $ and $A_{n}^{-1}C_{n}\left[ 
\begin{array}{c}
\theta \\ 
0%
\end{array}%
\right] \allowbreak \allowbreak =\allowbreak \left[ 
\begin{array}{c}
\sigma \theta (q^{n}+2[n]_{q}) \\ 
0%
\end{array}%
\right] $. Hence vector $\left[ 
\begin{array}{c}
\gamma _{n} \\ 
\delta _{n}%
\end{array}%
\right] $ satisfies the following recursion: $\left[ 
\begin{array}{c}
\gamma _{n+1} \\ 
\delta _{n+1}%
\end{array}%
\right] \allowbreak =\allowbreak \left[ 
\begin{array}{cc}
q & \sigma (1+q)q^{n} \\ 
0 & q%
\end{array}%
\right] \left[ 
\begin{array}{c}
\gamma _{n} \\ 
\delta _{n}%
\end{array}%
\right] \allowbreak +\allowbreak \left[ 
\begin{array}{c}
\sigma \theta (q^{n}+2[n]_{q}) \\ 
\theta%
\end{array}%
\right] $. So $\delta _{n}\allowbreak =\allowbreak \lbrack n]_{q}\theta $
and sequence $\gamma _{n}$ satisfies recursion: 
\begin{equation*}
\gamma _{n+1}\allowbreak =\allowbreak q\gamma _{n}+\sigma \theta
(1+q)q^{n}[n]_{q}+\sigma \theta (q^{n}+2[n]_{q}).
\end{equation*}%
One can easily check that 
\begin{eqnarray*}
&&[n+1]_{q}(q^{n}+2[n]_{q})\allowbreak -\allowbreak
q[n]_{q}(q^{n-1}+2[n-1]_{q}) \\
&=&q^{n}(1+q)[n]_{q}+q^{n}+2[n]_{q},
\end{eqnarray*}%
since we have $[n+1]_{q}(q^{n}+2[n]_{q})\allowbreak -\allowbreak
q[n]_{q}(-q^{n-1}+2[n]_{q})\allowbreak \allowbreak =\allowbreak
2[n]_{q}([n+1]_{q}\allowbreak -\allowbreak q[n]_{q})\allowbreak +\allowbreak
q^{n}[n+1]_{q}+q^{n}[n]_{q}\allowbreak =\allowbreak
2[n]_{q}+q^{n}\allowbreak +\allowbreak q^{n}[n]_{q}\allowbreak +\allowbreak
q^{n}([n+1]_{q}-1]\allowbreak =\allowbreak \lbrack
n]_{q}+[n+1]_{q}\allowbreak +\allowbreak q^{n}(1+q)[n]_{q}$. Hence we deduce
that $\gamma _{n}\allowbreak =\allowbreak \lbrack n]_{q}([n]_{q}\allowbreak
+\allowbreak \lbrack n-1]_{q})\theta \sigma $ by direct checking and
uniqueness of the solution. Using these results we can write recursion to be
satisfied by sequence $\chi _{n}$:%
\begin{eqnarray*}
\chi _{n+1} &=&q\chi _{n}+1+\theta ^{2}\sigma \lbrack
n]_{q}([n]_{q}+[n-1]_{q})+\sigma \theta ^{2}[n]_{q}^{2}-(1-q)\theta
^{2}\sigma \lbrack n]_{q}^{2}([n]_{q}+[n-1]_{q}) \\
&=&q\chi _{n}+1+\theta ^{2}\sigma \lbrack
n]_{q}(2[n]_{q}+[n-1]_{q}-(1-q^{n})([n]_{q}+[n-1]_{q})) \\
&=&q\chi _{n}+1+\theta ^{2}\sigma \lbrack
n]_{q}([n]_{q}+q^{n}([n]_{q}+[n-1]_{q})).
\end{eqnarray*}%
Let us denote $\zeta _{n}\allowbreak =\allowbreak \chi _{n}-[n]_{q}$. We see
that sequence $\zeta _{n}$ satisfies the following recursion:%
\begin{equation*}
\zeta _{n+1}\allowbreak =\allowbreak q\zeta _{n}+\theta ^{2}\sigma \lbrack
n]_{q}([n]_{q}+q^{n}([n]_{q}+[n-1]_{q})).
\end{equation*}%
Let us notice that 
\begin{eqnarray*}
&&[n+1]_{q}[n]_{q}^{2}\allowbreak -\allowbreak q[n]_{q}[n-1]_{q}^{2} \\
&=&[n]_{q}([n]_{q}\allowbreak +\allowbreak qq^{n-1}([n]_{q}\allowbreak
+\allowbreak \lbrack n-1]_{q}),
\end{eqnarray*}%
since $\allowbreak \lbrack
n]_{q}([n+1]_{q}[n]_{q}-q[n-1]_{q}^{2})\allowbreak =\allowbreak \lbrack
n]_{q}([n]_{q}+q[n]_{q}^{2}-q[n-1]_{q}^{2})\allowbreak \allowbreak $. Hence
we see that $\zeta _{n}\allowbreak =\allowbreak \lbrack
n]_{q}[n-1]_{q}^{2}\theta ^{2}\sigma $. Similarly we show the other
statement of this assertion.

iii) Under our assumptions we have $A_{n}\allowbreak =\allowbreak I,$ $%
B_{n}\allowbreak =\allowbreak \left[ 
\begin{array}{cc}
q & 0 \\ 
0 & q%
\end{array}%
\right] ,$ $C_{n}\allowbreak =\allowbreak \left[ 
\begin{array}{cc}
0 & 1 \\ 
1 & 0%
\end{array}%
\right] $. Vector $\left[ 
\begin{array}{c}
\gamma _{n} \\ 
\delta _{n}%
\end{array}%
\right] $ in fact satisfies two separate equations : $\gamma
_{n+1}\allowbreak =\allowbreak q\gamma _{n}+\eta $ which results in $\gamma
_{n}\allowbreak =\allowbreak \eta \lbrack n]_{q}$ and $\delta
_{n+1}\allowbreak =\allowbreak q\delta _{n}+\theta .$ Which results in $%
[n]_{q}\theta $. Now Inserting these quantities to equation (\ref{be})
yields the following recursion:%
\begin{equation*}
\chi _{n+1}\allowbreak =\allowbreak q\chi _{n}+1+\theta \eta \lbrack
n]_{q}(1+q^{n}).
\end{equation*}%
Again we have 
\begin{eqnarray*}
&&[n+1]_{q}[n]_{q}\allowbreak -\allowbreak q[n]_{q}[n-1]_{q} \\
&=&[n]_{q}(1+q^{n}).
\end{eqnarray*}%
$\allowbreak \allowbreak $ Thus we deduce that $\chi _{n}\allowbreak
=\allowbreak \lbrack n]_{q}\allowbreak +\allowbreak \theta \eta \lbrack
n]_{q}[n-1]_{q}$.

iv) Assumption that $\sigma \allowbreak =\allowbreak 0$ implies that $%
\lambda _{n}\allowbreak =\allowbreak 1$. Thus we have $A_{n}\allowbreak
=\allowbreak \left[ 
\begin{array}{cc}
1 & 0 \\ 
-\tau & 1%
\end{array}%
\right] ,$ $B_{n}\allowbreak =\allowbreak \left[ 
\begin{array}{cc}
0 & 0 \\ 
0 & 0%
\end{array}%
\right] ,$ $C_{n}\allowbreak =\allowbreak \left[ 
\begin{array}{cc}
0 & 1 \\ 
1 & \tau%
\end{array}%
\right] $. Hence $\gamma _{n}$ and $\delta _{n}$ do not depend on $n$ and it
is elementary that they are equal to $\eta $ and $\allowbreak \theta +2\eta
\tau $ respectively. Further we have $\left. (q+\sigma \tau -\sigma \tau
(1-\lambda _{n})^{2})\right\vert _{\sigma =0,q=0}\allowbreak =\allowbreak 0$
and $\left. (1-\sigma \tau (2\lambda _{n}+q\lambda _{n}^{2}))\right\vert
_{\sigma =0,q=0}\allowbreak =\allowbreak 1$ and $\theta \gamma
_{n}\allowbreak +\allowbreak \eta \delta _{n}\allowbreak +\allowbreak \tau
\gamma _{n}^{2}\allowbreak +\allowbreak \sigma \delta _{n}^{2}\allowbreak
-\allowbreak (1-q)\gamma _{n}\delta _{n}\allowbreak +\allowbreak
1\allowbreak =\allowbreak 1+\theta \eta \allowbreak +\allowbreak \tau \eta
^{2}$. In case $\tau \allowbreak =\allowbreak q\allowbreak =\allowbreak 1$
we argue in the similar way.

v) Under this assumption equation (\ref{_lambda}) reduces to 
\begin{equation*}
\lambda _{n+1}\allowbreak =\allowbreak \frac{1-\sigma \tau \lambda _{n}}{%
1-\sigma \tau \lambda _{n}}\allowbreak =\allowbreak 1.
\end{equation*}%
Besides we have for $n\geq 1:$ $A_{n}\allowbreak =\allowbreak \left[ 
\begin{array}{cc}
1-\tau \sigma & -\sigma (1-\sigma \tau ) \\ 
-\tau (1-\sigma \tau ) & 1-\sigma \tau%
\end{array}%
\right] \allowbreak ,\allowbreak B_{n}\allowbreak =\allowbreak \left[ 
\begin{array}{cc}
0 & 0 \\ 
0 & 0%
\end{array}%
\right] ,$ $C_{n}\allowbreak =\allowbreak C_{n}=\left[ 
\begin{array}{cc}
\sigma & 1 \\ 
1 & \tau%
\end{array}%
\right] $ , so $\Xi _{n}\allowbreak =\allowbreak A_{n}^{-1}B_{n}\allowbreak
=\allowbreak \left[ 
\begin{array}{cc}
0 & 0 \\ 
0 & 0%
\end{array}%
\right] \overset{df}{=}\Xi $. Further $\prod_{k=1}^{n}\Xi _{k}\allowbreak
=0, $ $w_{n}\allowbreak =\allowbreak \frac{1}{\left( 1-\sigma \tau \right)
^{2}}\left[ 
\begin{array}{c}
2\theta \sigma +\eta (1+\sigma \tau ) \\ 
2\eta \tau +\theta (1+\sigma \tau )%
\end{array}%
\right] $. So $\left[ 
\begin{array}{c}
\gamma _{n} \\ 
\delta _{n}%
\end{array}%
\right] \allowbreak =\frac{1}{\left( 1-\sigma \tau \right) ^{2}}\left[ 
\begin{array}{c}
2\theta \sigma +\eta (1+\sigma \tau ) \\ 
2\eta \tau +\theta (1+\sigma \tau )%
\end{array}%
\right] \allowbreak $ for $n\geq 0$. Besides $\left. (q+\sigma \tau -\sigma
\tau (1-\lambda _{n})^{2})\right\vert _{q+\sigma \tau =0}\allowbreak
=\allowbreak 0$ and $\left. (1-\sigma \tau (2\lambda _{n}+q\lambda
_{n}^{2}))\right\vert _{q+\sigma \tau =0}\allowbreak =\allowbreak (1-\sigma
\tau )^{2}$. Since $\gamma _{n}$ and $\delta _{n}$ do not depend on $n$ and
we have: $\theta \gamma _{n}\allowbreak +\allowbreak \eta \delta
_{n}\allowbreak =\allowbreak \frac{2\left( \theta +\tau \eta \right) \left(
\eta +\theta \sigma \right) }{\left( 1-\sigma \tau \right) ^{2}}$ and $\tau
\gamma _{n}^{2}\allowbreak +\allowbreak \sigma \delta _{n}^{2}\allowbreak
-\allowbreak (1-q)\gamma _{n}\delta _{n}\allowbreak =\allowbreak \frac{%
-\left( \theta +\tau \eta \right) \left( \eta +\theta \sigma \right) }{%
\left( 1-\sigma \tau \right) ^{2}}$. Hence we deduce we deduce that $\beta
_{n-1}\varepsilon _{n}$ also does not depend on $n$. By direct calculation
we have $\chi _{1}\allowbreak =\allowbreak \frac{1}{\left( 1-\sigma \tau
\right) ^{2}},$ while for $n>1$ we have: $1\allowbreak +\allowbreak \theta
\gamma _{n}\allowbreak +\allowbreak \eta \delta _{n}\allowbreak +\allowbreak
\tau \gamma _{n}^{2}\allowbreak +\allowbreak \sigma \delta _{n}\allowbreak
-\allowbreak (1-q)\gamma _{n}\delta _{n}\allowbreak =\allowbreak
1\allowbreak +\allowbreak \frac{\left( \theta +\tau \eta \right) \left( \eta
+\theta \sigma \right) }{\left( 1-\sigma \tau \right) ^{2}}$.

vi) First of all notice that under our assumptions we have $q+\sigma \tau
\allowbreak =\allowbreak (1-\sqrt{\sigma \tau })^{2}$. Next notice that if $%
n\allowbreak =\allowbreak 1$ then $\lambda _{1}\allowbreak =\allowbreak
1\allowbreak =\allowbreak \left. \frac{n}{1+(n-1)\sqrt{\sigma \tau }}%
\right\vert _{n=1}$. Hence by induction we have $1+(1-2\sqrt{\sigma \tau }%
)n/(1+(n-1)\sqrt{\sigma \tau }))\allowbreak =\allowbreak \allowbreak \frac{%
(1-\sqrt{\sigma \tau })(n+1)}{1+(n-1)\sqrt{\sigma \tau }}$ and $1-\sigma
\tau n/(1+(n-1)\sqrt{\sigma \tau })\allowbreak =\allowbreak \frac{(1-\sqrt{%
\sigma \tau })(1+n\sqrt{\sigma \tau })}{(1+(n-1)\sqrt{\sigma \tau })}$. Thus 
\begin{equation*}
\lambda _{n+1}\allowbreak =\allowbreak \left. (1+q\lambda _{n})/(1-\sigma
\tau \lambda _{n})\right\vert _{\lambda _{n}=n/(1+(n-1)\sqrt{\sigma \tau }%
)}\allowbreak =\allowbreak \frac{n+1}{1+n\sqrt{\sigma \tau }}.
\end{equation*}%
Now notice that $\left. q+\sigma \tau -\sigma \tau (1-\lambda
_{n-1})^{2}\right\vert _{q=1-2\sqrt{\sigma \tau }}\allowbreak =\allowbreak
(1-\sqrt{\sigma \tau })^{2}-\sigma \tau (1-\frac{n-1}{1+(n-2)\sqrt{\sigma
\tau }})^{2}\allowbreak =\allowbreak \frac{(1-\sqrt{\sigma \tau }%
)^{2}(1+2(n-2)\sqrt{\sigma \tau })}{(1+(n-2)\sqrt{\sigma \tau })^{2}}$ $%
\allowbreak $ and $\left. (1-\sigma \tau (2\lambda _{n}+q\lambda
_{n}^{2}))\right\vert _{q=1-2\sqrt{\sigma \tau }}\allowbreak =\allowbreak
\allowbreak \frac{(1-\sqrt{\sigma \tau })^{2}(1+2n\sqrt{\sigma \tau })}{%
(1+(n-1)\sqrt{\sigma \tau })^{2}}$. So sequence $\left\{ \chi _{n}\right\} $
satisfies the following recursion:%
\begin{equation}
\chi _{n+1}=\frac{(1+2(n-2)\sqrt{\sigma \tau })(1+(n-1)\sqrt{\sigma \tau }%
)^{2}}{(1+2n\sqrt{\sigma \tau })(1+(n-2)\sqrt{\sigma \tau })^{2}}\chi _{n}+%
\frac{(1+(n-1)\sqrt{\sigma \tau })^{2}}{(1-\sqrt{\sigma \tau })^{2}(1+2n%
\sqrt{\sigma \tau })}.  \label{rec1}
\end{equation}%
The proof is now by induction setting $n=1$ in (\ref{q+2p2}) we get $1$.
Further by direct calculation we have:%
\begin{gather*}
\frac{n(1+(n-2)\sqrt{\sigma \tau })^{2}(1+(n-3)\sqrt{\sigma \tau })}{(1-%
\sqrt{\sigma \tau })^{2}(1+2(n-1)\sqrt{\sigma \tau })(1+2(n-2)\sqrt{\sigma
\tau })}\times \frac{(1+2(n-2)\sqrt{\sigma \tau })(1+(n-1)\sqrt{\sigma \tau }%
)^{2}}{(1+2n\sqrt{\sigma \tau })(1+(n-2)\sqrt{\sigma \tau })^{2}} \\
+\frac{(1+(n-1)\sqrt{\sigma \tau })^{2}}{(1-\sqrt{\sigma \tau })^{2}(1+2n%
\sqrt{\sigma \tau })}=\frac{(n+1)(1+(n-1)\sqrt{\sigma \tau })^{2}(1+(n-2)%
\sqrt{\sigma \tau })}{(1-\sqrt{\sigma \tau })^{2}(1+2n\sqrt{\sigma \tau }%
)(1+2(n-1)\sqrt{\sigma \tau })}.
\end{gather*}
\end{proof}

\appendix{}

\section{Quadratic Harnesses\label{QH}}

Let us recall, following \cite{BryMaWe07} that QH is a stochastic process $%
\left\{ X_{t}\right\} _{t\geq 0}$ defined for $t\geq 0$ on a certain
probability space $(\Omega ,\mathcal{F},P)$ satisfying the following
definition:

\begin{definition}
\label{defincja}A stochastic process $\left\{ X_{t}\right\} _{t\geq 0}$ will
be called quadratic harness if the following $4$ conditions are satisfied:

1. $X_{0}\allowbreak =\allowbreak 0,$ $\forall t\geq 0,~EX_{t}\allowbreak
=\allowbreak 0,$

2. $\forall ~s,t\geq 0,$ $EX_{s}X_{t}\allowbreak =\allowbreak \min (s,t),$

3. $\forall ~0\leq s<t<u:$ $E\left( X_{t}|\mathcal{F}_{s,u}\right)
\allowbreak =\allowbreak \frac{u-t}{u-s}X_{s}\allowbreak +\allowbreak \frac{%
t-s}{u-s}X_{u},$ a.s.

4. $\forall ~0\leq s<t<u:$ $E\left( X_{t}^{2}|\mathcal{F}_{s,u}\right)
\allowbreak =\allowbreak Q_{s,t,u}\left( X_{s},X_{u}\right) ,$

where $Q_{s,t,u}\left( x,y\right) $ is a certain quadratic form determined
by $6$ coefficients and $\mathcal{F}_{s,u}\allowbreak =\allowbreak \sigma
(X_{t}:t\in (0,s]\cup \lbrack u,\infty ))$.
\end{definition}

Bryc , Matysiak, Weso\l owski showed in \cite{BryMaWe07} that there exist $5$
parameters which they denoted by $\tau ,\sigma ,\theta ,\eta ,q$ such that
the quadratic form $Q$ is completely determined i.e. respective coefficients
are defined by the known functions of $s,t,u$ and $\tau ,\sigma ,\theta
,\eta ,q$. Bryc, Matysiak, Weso\l owski deduced that $\sigma ,\tau \geq 0,$ $%
q\leq 1+2\sqrt{\sigma \tau }$ and $\eta ,\theta \in \mathbb{R}$. More
precisely they showed that%
\begin{equation*}
Q_{s,t,u}\left( x,y\right) \allowbreak =\allowbreak A\left( s,t,u\right)
x^{2}+B\left( s,t,u\right) xy+C\left( s,t,u\right) y^{2}+D\left(
s,t,u\right) x+E\left( s,t,u\right) y+F\left( s,t,u\right) ,
\end{equation*}%
where 
\begin{eqnarray}
A\left( s,t,u\right) &=&\frac{(u-t)(u(1+\sigma t)+\tau -qt)}{%
(u-s)(u(1+\sigma s)+\tau -qs)},  \label{_A} \\
B\left( s,t,u\right) &=&\frac{(u-t)(t-s)(1+q)}{(u-s)(u(1+\sigma s)+\tau -qs)}%
,  \label{_B} \\
C\left( s,t,u\right) &=&\frac{(t-s)(t(1+\sigma s)+\tau -qs)}{%
(u-s)(u(1+\sigma s)+\tau -qs)},  \label{_C} \\
D\left( s,t,u\right) &=&\frac{(u-t)(t-s)(u\eta -\theta )}{(u-s)(u(1+\sigma
s)+\tau -qs)},  \label{_D} \\
E\left( s,t,u\right) &=&\frac{(u-t)(t-s)(-s\eta +\theta )}{(u-s)(u(1+\sigma
s)+\tau -qs)},  \label{_E} \\
F(s,t,u) &=&\frac{(u-t)(t-s)}{(u(1+\sigma s)+\tau -qs)}.  \label{_F}
\end{eqnarray}

The authors were seeking quadratic harnesses that are also Markov processes
and assuming the existence of all moments they were trying to find a family
of orthogonal polynomials $\left\{ p_{n}\left( x;t\right) \right\} _{t\geq
0,n\geq -1}$ such that%
\begin{equation}
\forall n\geq 0,t>s\geq 0:E\left( p_{n}\left( X_{t};t\right) |\mathcal{F}%
_{\leq s}\right) \allowbreak =\allowbreak p_{n}\left( X_{s};s\right) ,\text{%
a.s.}  \label{QMP}
\end{equation}%
Such family of QH that are also Markov will be called MQH and obviously they
constitute a subset of all QH.

Family of orthogonal polynomials of MQH will be called orthogonal martingale
polynomials (briefly OM family of polynomials of the MQH $\left\{
X_{t}\right\} $).

Obviously we have $p_{-1}\left( x,t\right) \allowbreak =\allowbreak 0,$ $%
p_{0}\left( x;t\right) \allowbreak =\allowbreak 1$. Moreover the authors
show in \cite{BryMaWe07} that $p_{1}\left( x;t\right) \allowbreak
=\allowbreak x$. Now recall that following general theory of orthogonal
polynomials presented e.g. in \cite{IA} or in \cite{Koek} that every family
of orthogonal polynomials $\left\{ r_{n}\left( x\right) \right\} $ satisfies
the so called 3-term recurrence (\ref{3tr}), i.e. the product $xr_{n}\left(
x\right) $ is a linear combination of $r_{k}$ for $k\allowbreak =\allowbreak
n+1,n,n-1$.

Note that if (\ref{3tr}) and (\ref{QMP}) are to make sense we must have $%
a_{n}\left( t\right) >0$ for all $t$ and $n>-1$. Moreover from the general
theory of orthogonal polynomials it follows that if $a_{n-1}\left( t\right)
c_{n}\left( t\right) \geq 0$ for all $n$ then the measure with respect to
which polynomials $p_{n}$ are to be orthogonal is nonnegative i.e.
polynomials have probabilistic interpretation. Hence it is reasonable to
consider only such QH for which this condition is satisfied for all $n>-1$
and $t\geq 0$.

Bryc, Matysiak, Weso\l owski showed also in the same paper that coefficients 
$a_{n},$ $b_{n},$ $c_{n}$ must be linear functions of $t$. This an easy
conclusion of the condition 3. of the Definition \ref{defincja}. For the
sake of completeness we will prove this fact.

\begin{proposition}
\label{1_dim}Let $\left\{ X_{t}\right\} _{t\geq 0}$ be MQH with parameters $%
\sigma ,\tau ,\theta ,\eta ,q$ such that $\forall t>0:$ $\limfunc{supp}X_{t}$
contains infinite number of points. Let $\left\{ p_{n}\left( x;t\right)
\right\} _{n\geq 0}$ denote its family of OM polynomials. Then

i) $\forall n>0$ $p_{n}(0,0)\allowbreak =\allowbreak 0$ consequently $%
\forall n>0:$ $E(p_{n}(X_{t};t))\allowbreak =\allowbreak 0$ thus polynomials 
$p_{n}$ constitute the family of orthogonal polynomials of the marginal
distribution i.e. distribution of $X_{t}$.

ii) There must exist six number sequences $\left\{ \alpha _{n}\right\} ,$ $%
\left\{ \beta _{n}\right\} ,$ $\left\{ \gamma _{n}\right\} ,$ $\left\{
\delta _{n}\right\} ,$ $\left\{ \varepsilon _{n}\right\} ,$ $\left\{ \varphi
_{n}\right\} $ such that:%
\begin{equation*}
a_{n}\left( t\right) =\alpha _{n}t+\beta _{n},~b_{n}\left( t\right)
\allowbreak =\allowbreak \gamma _{n}t+\delta _{n},~c_{n}\left( t\right)
=\varepsilon _{n}t+\varphi _{n},
\end{equation*}%
with $\alpha _{0}=0,$ $\beta _{0}=1,$ $\gamma _{0}=0,\delta _{0}=0$, $%
\varepsilon _{1}=1,$ $\varphi _{1}=0$.
\end{proposition}

\begin{proof}[Proof of Proposition \protect\ref{1_dim}]
i) First of all notice that from (\ref{QMP}) it follows that $\forall n>-1$ $%
Ep_{n}(X_{t},t)\allowbreak =\allowbreak Ep_{n}\left( 0,0\right) \allowbreak
=\allowbreak \xi _{n}$ a constant that does not depend on $t$. Secondly
notice that $a_{n}(0)\allowbreak =\allowbreak \beta _{n},$ $b_{n}(0)=\delta
_{n},$ $c_{n}(0)=\varphi _{n}$. Further notice that following (\ref{3tr})
these constants satisfy the following second order recursion:%
\begin{equation*}
\xi _{n+1}\allowbreak =\allowbreak -\frac{\delta _{n}}{\beta _{n}}\xi _{n}-%
\frac{\varphi _{n}}{\beta _{n}}\xi _{n-1},
\end{equation*}%
with $\xi _{-1}\allowbreak =\allowbreak 0,$ $\xi _{0}\allowbreak
=\allowbreak 1$. Besides we also have $0\allowbreak =\allowbreak \beta
_{0}p_{1}\left( 0;0\right) +\delta _{0}p_{0}\left( 0;0\right) +\varphi
_{0}p_{-1}\left( 0;0\right) $. Hence we deduce that $p_{1}(0,0)\allowbreak
=\allowbreak 0,$ that is $\xi _{1}\allowbreak =\allowbreak 0$. Now notice
that if we chose $\varphi _{1}\allowbreak =\allowbreak 0$ then we would have 
$\xi _{2}\allowbreak =\allowbreak 0$ that is two successive constants $\xi
_{n}$ being equal to zero consequently all must be equal to zero. Thus the
choice $\varphi _{1}\allowbreak =\allowbreak 0$ enables to select sequence $%
\left\{ p_{n}\right\} $ to be both OM and have the property that $%
Ep_{n}(X_{t},t)\allowbreak =\allowbreak 0$. On the other hand since $%
EX_{t}^{2}\allowbreak =\allowbreak t$ we take $n\allowbreak =\allowbreak 1$
in (\ref{3tr}) and use the fact that $Ep_{n}(X_{t},t)\allowbreak
=\allowbreak 0$ and deducing that $\varepsilon _{1}\allowbreak =\allowbreak
1 $. Sequence $\left\{ p_{n}\right\} $ is thus a sequence of orthogonal
polynomials that for some measure $\mu $ satisfy $\int p_{n}d\mu \allowbreak
=\allowbreak 0$ for all $n>0$. Since we have also 3-term recurrence
satisfied by polynomials $p_{n}$ we deduce that also $\int xp_{n}d\mu
\allowbreak =\allowbreak 0$ for all $n>1$. Similarly we deduce that $\int
x^{k}p_{n}d\mu \allowbreak =\allowbreak 0$ for all $n>k$. Hence polynomials
must constitute family of orthogonal polynomials of measure $\mu $.

ii) On one hand we have: $E(X_{t}p_{n}(X_{t};t)|\mathcal{B}_{\leq
s})\allowbreak =\allowbreak a_{n}(t)p_{n+1}(X_{s};s)\allowbreak +\allowbreak
b_{n}(t)p_{n}(X_{s};s)\allowbreak +\allowbreak c_{n}(t)p_{n-1}(X_{s};s)$. On
the other:%
\begin{gather*}
E(X_{t}p_{n}(X_{t};t)|\mathcal{B}_{\leq s})\allowbreak \allowbreak
=\allowbreak E(X_{t}p_{n}(X_{u};u)|\mathcal{B}_{\leq s})\allowbreak
=\allowbreak E(E(X_{t}|\mathcal{B}_{\leq s,\geq u})p_{n}(X_{u};u)|\mathcal{B}%
_{\leq s})\allowbreak = \\
\allowbreak \frac{(u-t)}{u-s}X_{s}p_{n}(X_{s};s)\allowbreak +\allowbreak 
\frac{t-s}{u-s}E\mathbb{(}X_{u}p_{n}(X_{u};u)|\mathcal{B}_{\leq
s})\allowbreak \\
=\allowbreak \frac{(u-t)}{u-s}(a_{n}(s)p_{n+1}(X_{s};s)\allowbreak
+\allowbreak b_{n}(s)p_{n}(X_{s};s)+c_{n}(s)p_{n-1}(X_{s};s))\allowbreak \\
+\allowbreak \frac{t-s}{u-s}%
(a_{n}(u)p_{n+1}(X_{s};s)+b_{n}(u)p_{n}(X_{s};s)+c_{n}(u)p_{n-1}(X_{s};s)).
\end{gather*}
Comparing appropriate coefficients we get: 
\begin{equation}
a_{n}(t)\allowbreak =\allowbreak \frac{(u-t)}{u-s}a_{n}(s)\allowbreak
+\allowbreak \frac{t-s}{u-s}a_{n}(u)  \label{_aa}
\end{equation}%
and similarly for $b_{n}(t)$ and $c_{n}\left( t\right) $. Now set $%
s\allowbreak =\allowbreak 0$ in (\ref{_aa}). Remembering that we had $%
a_{n}(0)\allowbreak =\allowbreak \beta _{0}$ and subtracting from both sides 
$\beta _{0}$ we get 
\begin{equation*}
\frac{a_{n}(t)-\beta _{0}}{t}=\frac{a_{n}(u)-\beta _{0}}{u},
\end{equation*}%
from which we deduce that $a_{n}(t)$ must be linear function of $t$. We
argue similarly for other coefficients $b_{n}(t)$ and $c_{n}(t)$.
\end{proof}

Hence these coefficients are defined in fact by $6$ families of sequences.
More precisely we will seek relationships between families of numbers that
are implied by the conditions that MQH's must satisfy.

We have the following theorem that is a version of the result of \cite%
{BryMaWe07}. More precisely since our definitions of the numbers parameters $%
\left\{ \alpha _{n}\right\} ,$ $\left\{ \beta _{n}\right\} ,$ $\left\{
\gamma _{n}\right\} ,$ $\left\{ \delta _{n}\right\} ,$ $\left\{ \varepsilon
_{n}\right\} ,$ $\left\{ \varphi _{n}\right\} $ is other than that the ones
from \cite{BryMaWe07} the resulting equations differ slightly but are
basically the same. We believe that our notation is more logical but as
usually it is in fact a matter of taste. Besides the proof that we are
presenting is much simpler both conceptually and mathematically than the
result of in \cite{BryMaWe07}. Nevertheless it is relatively long.

\begin{theorem}
\label{rownania} Assuming that process $\left\{ X_{t}\right\} _{t\geq 0}$ is
a MQH with parameters $\sigma ,\tau ,\theta ,\eta ,q$ and family of
polynomials $\left\{ p_{n}\right\} $ constitute its family of OM
polynomials. Then families of numbers $\left\{ \alpha _{n}\right\} ,$ $%
\left\{ \beta _{n}\right\} ,$ $\left\{ \gamma _{n}\right\} ,$ $\left\{
\delta _{n}\right\} ,$ $\left\{ \varepsilon _{n}\right\} ,$ $\left\{ \varphi
_{n}\right\} $ satisfy the following system of $5$ recurrences given by (\ref%
{_1})-(\ref{_6}).
\end{theorem}

\begin{proof}[Proof of Theorem \protect\ref{rownania}]
First of all notice that starting from : $xp_{n}(x;t)=(\alpha _{n}t+\beta
_{n})p_{n+1}(x;t)+(\gamma _{n}t+\delta _{n})p_{n}(x;t)+(\varepsilon
_{n}t+\varphi _{n})p_{n-1}(x;t)$. Iterating it we get: 
\begin{gather*}
x^{2}p_{n}(x;t)= \\
\allowbreak (\alpha _{n}t+\beta _{n})(\alpha _{n+1}t+\beta
_{n+1})p_{n+2}(x;t)\allowbreak +\allowbreak ((\alpha _{n}t+\beta
_{n})(\gamma _{n+1}t+\delta _{n+1}) \\
+(\gamma _{n}t+\delta _{n})(\alpha _{n}t+\beta _{n}))p_{n+1}(x;t) \\
\allowbreak +\allowbreak ((\alpha _{n}t+\beta _{n})(\varepsilon
_{n+1}t+\varphi _{n+1})+(\gamma _{n}t+\delta _{n})(\gamma _{n}t+\delta
_{n})\allowbreak \\
+\allowbreak (\varepsilon _{n}t+\varphi _{n})(\alpha _{n-1}t+\beta
_{n-1}))p_{n}(x;t)\allowbreak \\
+\allowbreak ((\gamma _{n}t+\delta _{n})(\varepsilon _{n}t+\varphi
_{n})\allowbreak +\allowbreak (\varepsilon _{n}t+\varphi _{n})(\gamma
_{n-1}t+\delta _{n-1}))p_{n-1}(x;t)\allowbreak \\
+\allowbreak (\varepsilon _{n}t+\varphi _{n})(\varepsilon _{n-1}t+\varphi
_{n-1})p_{n-2}\left( x;t\right) .
\end{gather*}

On one hand we get: 
\begin{gather*}
E\mathbb{(}X_{t}^{2}p_{n}(X_{t};t)|\mathcal{B}_{\leq s})\allowbreak
=\allowbreak (\alpha _{n}t+\beta _{n})(\alpha _{n+1}t+\beta
_{n+1})p_{n+2}(X_{s};s) \\
((\alpha _{n}t+\beta _{n})(\gamma _{n+1}t+\delta _{n+1})+(\gamma
_{n}t+\delta _{n})(\alpha _{n}t+\beta _{n}))p_{n+1}(X_{s};s)\allowbreak \\
+((\alpha _{n}t+\beta _{n})(\varepsilon _{n+1}t+\varphi _{n+1})+(\gamma
_{n}t+\delta _{n})(\gamma _{n}t+\delta _{n})\allowbreak +\allowbreak
(\varepsilon _{n}t+\varphi _{n})(\alpha _{n-1}t+\beta _{n-1}))p_{n}(X_{s};s)
\\
+\allowbreak ((\gamma _{n}t+\delta _{n})(\varepsilon _{n}t+\varphi
_{n})\allowbreak +\allowbreak (\varepsilon _{n}t+\varphi _{n})(\gamma
_{n-1}t+\delta _{n-1}))p_{n-1}(X_{s};s)\allowbreak \\
+(\varepsilon _{n}t+\varphi _{n})(\varepsilon _{n-1}t+\varphi
_{n-1})p_{n-2}\left( X_{s};s\right) ,
\end{gather*}%
while on the other we get:

\begin{gather*}
E\mathbb{(}X_{t}^{2}p_{n}(X_{t};t)|\mathcal{B}_{\leq s})\allowbreak
=\allowbreak E\mathbb{(}X_{t}^{2}p_{n}(X_{u};u)|\mathcal{B}_{\leq
s})\allowbreak =\allowbreak \allowbreak E\mathbb{(}E(X_{t}^{2}|\mathcal{B}%
_{\leq s,\geq u})p_{n}(X_{u};u)|\mathcal{B}_{\leq s})\allowbreak \\
=A(s,t,u)X_{s}^{2}p_{n}(X_{s};s)\allowbreak +\allowbreak
B(s,t,u)X_{s}((\alpha _{n}u+\beta _{n})p_{n+1}(X_{s};s) \\
+(\gamma _{n}u+\delta _{n})p_{n}(X_{s};s)+(\varepsilon _{n}u+\varphi
_{n})p_{n-1}(X_{s};s)) \\
+C(s,t,u)((\alpha _{n}u+\beta _{n})(\alpha _{n+1}u+\beta
_{n+1})p_{n+2}(X_{s};s) \\
+((\alpha _{n}u+\beta _{n})(\gamma _{n+1}u+\delta _{n+1})+(\gamma
_{n}u+\delta _{n})(\alpha _{n}u+\beta _{n}))p_{n+1}(X_{s};s) \\
+((\alpha _{n}u+\beta _{n})(\varepsilon _{n+1}u+\varphi _{n+1})+(\gamma
_{n}u+\delta _{n})(\gamma _{n}u+\delta _{n})\allowbreak \\
+\allowbreak (\varepsilon _{n}u+\varphi _{n})(\alpha _{n-1}u+\beta
_{n-1}))p_{n}(X_{s};s) \\
+((\gamma _{n}u+\delta _{n})(\varepsilon _{n}u+\varphi _{n})\allowbreak
+\allowbreak (\varepsilon _{n}u+\varphi _{n})(\gamma _{n-1}u+\delta
_{n-1}))p_{n-1}(X_{s};s) \\
+(\varepsilon _{n}u+\varphi _{n})(\varepsilon _{n-1}u+\varphi
_{n-1})p_{n-2}\left( X_{s};s\right) ) \\
+D(s,t,u)X_{s}p_{n}(X_{s};s)\allowbreak +\allowbreak E(s,t,u)((\alpha
_{n}u+\beta _{n})p_{n+1}(X_{s};s) \\
+(\gamma _{n}u+\delta _{n})p_{n}(X_{s};s)+(\varepsilon _{n}u+\varphi
_{n})p_{n-1}(X_{s};s)+F(s,t,u)p_{n}(X_{s};s).
\end{gather*}%
Comparing coefficients by respectively $p_{n+i}(X_{s};s),$ $i=2,-2,1,-1,0$
we get: 
\begin{eqnarray*}
0 &=&(\alpha _{n}t+\beta _{n})(\alpha _{n+1}t+\beta _{n+1})\allowbreak
-\allowbreak A(s,t,u)(\alpha _{n}s+\beta _{n})(\alpha _{n+1}s+\beta _{n+1})
\\
&&-B(s,t,u)(\alpha _{n}u+\beta _{n})(\alpha _{n+1}s+\beta
_{n})-C(s,t,u)(\alpha _{n}u+\beta _{n})(\alpha _{n+1}u+\beta _{n+1})
\end{eqnarray*}%
from which it follows that (by Mathematica):%
\begin{equation*}
\tau \alpha _{n}\alpha _{n+1}-\alpha _{n+1}\beta _{n}+q\alpha _{n}\beta
_{n+1}+\sigma \beta _{n}\beta _{n+1}=0.
\end{equation*}%
Similarly we get the remaining equations.
\end{proof}

\end{document}